\documentclass[article,11pt]{elsarticle}
\usepackage[left=2.5cm, right=2.5cm, top=2cm, bottom=2cm]{geometry}

\usepackage{amsmath, amsthm, amssymb}
\usepackage{graphicx}
\usepackage{tikz}
\usepackage{url}

\newtheorem{lemma}{Lemma}
\newtheorem{theorem}{Theorem}
\newtheorem{corollary}{Corollary}
\newtheorem{remark}{Remark}
\newtheorem{definition}{Definition}

\begin{document}

\begin{frontmatter}

\title{A Unified Linear Programming Approach for Maximizing Vertex-Degree-Based Indices of Trees}

\author[nis]{Milan Ba\v{s}i\'{c}\corref{cor1}}
\ead{basic_milan@yahoo.com}
\ead{milan.basic@pmf.edu.rs}

\cortext[cor1]{Corresponding author}

\address[nis]{Faculty of Sciences and Mathematics, University of Ni\v{s}, Vi\v{s}egradska 33, 18000 Ni\v{s}, Serbia}

\begin{abstract}
The problem of maximizing vertex-degree-based topological indices over trees with a prescribed order and maximum degree $\Delta$ has traditionally relied on ad-hoc combinatorial arguments. In this paper, we establish a unified linear programming approach that systematically resolves these extremal problems. By formulating the optimization over degree-pair multiplicities, the search for an extremal tree reduces to minimizing a discrete dual penalty. Treating vertex degrees as continuous variables yields an analytical single-variable test condition that guarantees global optimality.
We apply this framework across a diverse collection of invariants. We prove that the standard basic feasible solution is optimal for the Tangent Sombor, $SO_1$, $SO_2$, Albertson, and $\sigma$-irregularity indices for all $\Delta \ge 3$. However, the dual optimality conditions establish exact analytical thresholds at which the standard basis loses optimality, necessitating a transition to alternative basic feasible solutions. For the Arithmetic-Geometric ($AG$) index, the standard basis is optimal strictly for $\Delta \ge 9$; for smaller maximum degrees ($3 \le \Delta \le 8$), optimization shifts to an alternative edge basis, allowing us to explicitly construct and characterize all extremal trees across congruence classes, including chemical trees ($\Delta \in \{3, 4\}$). Conversely, for the Triangle Volume index ($SO_{11}$), the standard basis is optimal exclusively for $\Delta \le 8$, while for $\Delta \ge 9$ we analytically derive a novel, degree-dependent optimal basis to explicitly construct its extremal configurations.
\end{abstract}


\end{frontmatter}

\section{Introduction}
\label{sec:intro}

Degree-based topological indices are fundamental tools in chemical graph theory and combinatorial optimization, serving as quantitative descriptors of molecular structure. A generalized vertex-degree-based (VDB) index is defined as $TI(G) = \sum_{uv \in E} f(d_u, d_v)$, where $f(x,y)$ is a symmetric weight function applied to the degrees of adjacent vertices $u$ and $v$. Classical examples of such invariants include the Arithmetic-Geometric ($AG$) index \cite{ag_amc_reference} and the Albertson irregularity index \cite{albertson1997}. A related quantitative measure, the $\sigma$-irregularity index, serves as a quadratic analogue to Albertson's index and has been extensively studied for its extremal graph properties \cite{abdo2018sigma, gutman2018inverse}.

A geometry-motivated approach to defining VDB indices was introduced by Gutman \cite{gutman2021sombor}, who mapped the end-vertex degrees of an edge to a two-dimensional coordinate plane. This paradigm gave rise to the Sombor indices and initiated the development of numerous related structural descriptors \cite{rada2021general, ravindrana2025review}. Returning to this geometric construction, Gutman \cite{gutman2022back} expanded the framework by introducing six \emph{Sombor-index-like} invariants, denoted as $SO_1$ through $SO_6$.

Following this conceptualization, further research diverged into two main directions. The first direction focused on proposing new topological invariants derived from explicit geometric, trigonometric, and calculus-based properties within the coordinate model. By analyzing trigonometric relations, \"{O}z and Gutman \cite{oz2025more} defined the tangent Sombor index ($Tan$). Utilizing calculus and solids of revolution, \"{O}z \cite{oz2024novel} further extended this framework to three dimensions, introducing a series of five volume- and surface-area-based indices ($SO_7$ through $SO_{11}$), including the Triangle Volume index ($SO_{11}$).

The second direction concentrated on resolving extremal problems for these invariants. Tang et al. \cite{tang2023trees} initiated this study by posing open problems regarding the extremal values of Sombor-index-like invariants over general trees, molecular trees, and connected graphs. Subsequent works by Ali et al. \cite{ali2023solutions} and Gao \cite{gao2024extremal} addressed these fundamental questions. Building on this foundation, recent research has introduced a prescribed maximum degree $\Delta$ as a critical structural constraint, generalizing the natural restriction of chemical trees ($\Delta \le 4$). Exact maximum values and explicit structural characterizations were recently derived for the $\sigma$-irregularity index \cite{kovijanic2024chemical, kovijanic2026delta5, basic2026sigma_arxiv, knor2026complete}, as well as for the $SO_5$ and $SO_6$ invariants \cite{basic2026filomat, basic2026cmj} on trees with bounded maximum degree.

Despite the success of these approaches, the standard methodologies in the literature generally remain index-specific. Establishing an extremal tree typically involves localized structural manipulations---such as vertex-shifting or edge-lifting---which demand separate combinatorial proofs for each distinct weight function $f(x,y)$.

In this paper, we depart from index-specific combinatorial techniques and establish a unified linear programming approach. By formulating the optimization problem over degree-pair multiplicities, the search for an extremal tree reduces to minimizing a discrete dual penalty. By treating vertex degrees as continuous variables, we derive an analytical single-variable test condition that guarantees global optimality.

We apply this methodology to determine the maximal trees across a diverse collection of invariants. We prove that the standard basic feasible solution is optimal for the Tangent, $SO_1$, $SO_2$, Albertson, and $\sigma$-irregularity indices for all $\Delta \ge 3$. However, the dual optimality conditions establish exact analytical thresholds at which the standard basis loses optimality, necessitating a transition to alternative edge bases. For the $AG$ index, the standard basis remains optimal strictly for $\Delta \ge 9$; for $3 \le \Delta \le 8$, the optimization shifts to an alternative basis, allowing us to explicitly construct all extremal trees across congruence classes, including chemical trees ($\Delta \in \{3, 4\}$). Conversely, we resolve the extremal configuration for the Triangle Volume index ($SO_{11}$), proving that the standard basis is optimal exclusively for $\Delta \le 8$. For $\Delta \ge 9$, we analytically derive a novel, degree-dependent optimal basis to explicitly construct its extremal trees, demonstrating that the optimal internal structure scales asymptotically as $\Delta/3$.

The paper is organized as follows. Section 2 introduces the geometric representation and definitions of the considered indices. Section 3 formulates the continuous linear programming model over degree-pair multiplicities and derives the discrete dual penalty function alongside the single-variable optimality condition. Section 4 characterizes the corresponding bipartite structural class and proves the optimality of the standard basis for the Tangent, $SO_1$, $SO_2$, $AG$, $\sigma$-irregularity, and Albertson indices. Section 5 establishes an alternative optimality criterion for small maximum degrees, explicitly resolving the extremal configurations for the $AG$ index when $\Delta \in \{3, 4\}$. Finally, Section 6 derives the degree-dependent optimal basis and exact extremal trees for the $SO_{11}$ index.

\section{Geometric Representation and Derived Indices}
\label{sec:geometric_representation}

A generalized vertex-degree-based topological index $TI(G)$ of a graph $G$ can be expressed as the sum of individual edge contributions:
\begin{equation}
\label{Eq:General_TI}
TI(G) = \sum_{uv \in E} f(d_u, d_v),
\end{equation}
where $f(x,y)$ is a symmetric real-valued function known as the \textit{topological weight function}. In alignment with the optimization framework established in the previous section, we consistently evaluate $f(x,y)$ under the assumption $x \le y$.

From now on, let us assume that all graphs are simple graphs and that for any edge $uv \in E$, its end-vertex degrees are given by $d_u = x$ and $d_v = y$ with $x < y$. The geometric properties of the representation shown in Figure \ref{fig:geometric-edge} yield a rich variety of well-known topological indices.

\begin{figure}[ht]
    \centering
    \begin{tikzpicture}[scale=1.2]
        \draw[->, thick] (0,0) -- (4,0) node[below] {$x$};
        \draw[->, thick] (0,0) -- (0,4) node[left] {$y$};
        \node[below left] at (0,0) {$O$};
        \fill (0,0) circle (1.5pt);

        \coordinate (O) at (0,0);
        \coordinate (A) at (2,3); 
        \coordinate (B) at (3,2);
        \coordinate (C) at (2,0);
        \coordinate (D) at (3,0);

        \fill (A) circle (1.5pt) node[above right] {$A(x,y)$};
        \fill (B) circle (1.5pt) node[right] {$B(y,x)$};
        \fill (C) circle (1.5pt) node[below] {$C(x,0)$};
        \fill (D) circle (1.5pt) node[below] {$D(y,0)$};

        \draw[thick] (O) -- (B) node[midway, below right] {$l_1$};
        \draw[thick] (O) -- (A) node[midway, above left] {$l_2$};
        \draw[thick] (B) -- (A) node[midway, above right] {$l_3$};

        \draw[dashed] (B) -- (D);
        \draw[dashed] (A) -- (C);

        \draw (0.8,0.53) arc (33.69:56.31:0.96) node[midway, right, xshift=1pt, yshift=2pt] {$\alpha$};
    \end{tikzpicture}
    \caption{A geometric representation of an edge $e = uv$ of a graph $G$, connecting vertices $u$ and $v$ of degrees $d_u = x$ and $d_v = y$ ($x < y$). The point $A$ represents the edge $e = uv$, whereas the point $B$ is the respective dual. $O$ denotes the origin of the coordinate system.}
    \label{fig:geometric-edge}
\end{figure}
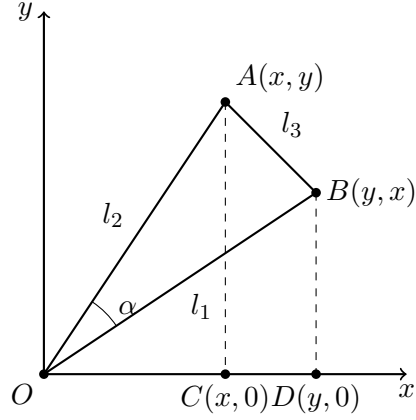

Based on this geometric representation, the side lengths of the triangle $\triangle AOB$ are given by $l_1 = \sqrt{y^2 + x^2}$, $l_2 = \sqrt{x^2 + y^2}$, and $l_3 = \sqrt{(x-y)^2 + (y-x)^2} = \sqrt{2}(y-x)$ (since $y > x$). The semiperimeter is $s = \frac{1}{2}(l_1 + l_2 + l_3) = \sqrt{x^2+y^2} + \frac{\sqrt{2}}{2}(y-x)$. The area of the triangle naturally evaluates to $area(AOB) = \frac{1}{2}(y^2 - x^2)$. 

We can also express trigonometric invariants based on the angle $\alpha$ between vectors $\vec{OA}$ and $\vec{OB}$:
\begin{align*}
    \sin \alpha = \frac{y^2 - x^2}{x^2 + y^2}, \qquad \cos \alpha = \frac{2xy}{x^2 + y^2}, \qquad \tan \alpha = \frac{y^2 - x^2}{2xy}.
\end{align*}

Furthermore, let us establish the radii of the circumscribed circle ($R$), the inscribed circle ($r$), and two of the excircles ($r_1, r_2$) corresponding to the sides $l_1$ and $l_2$ of $\triangle AOB$:
\begin{align*}
    R &= \frac{l_1 l_2 l_3}{4 \cdot area(AOB)} = \frac{\sqrt{2}(x^2+y^2)}{2(x+y)}, \\[1em]
    r &= \frac{area(AOB)}{s} = \frac{y^2-x^2}{2\sqrt{x^2+y^2} + \sqrt{2}(y-x)}, \\[1em]
    r_1 = r_2 &= \frac{area(AOB)}{s - l_1} = \frac{\sqrt{2}}{2}(x+y).
\end{align*}

We recall the definitions of $SO_1$ through $SO_4$ from \cite{oz2025more}, and additionally introduce the indices $SO_5$ and $SO_6$ corresponding to the inscribed circle of $\triangle AOB$:
\begin{equation}
\label{Eq:SO_indices_basic}
\begin{array}{ll}
	SO_1(G)=\displaystyle \sum_{uv\in E} area(AOB),
	&
	SO_2(G)=\displaystyle \sum_{uv\in E}\sin \alpha,
	\\[2em]
	SO_3(G)=\displaystyle \sum_{uv\in E}perim(\Gamma_c),
	&
	SO_4(G)=\displaystyle \sum_{uv\in E}area(\Gamma_c),
	\\[2em]
	SO_5(G)=\displaystyle \sum_{uv\in E}perim(\Gamma_I),
	&
	SO_6(G)=\displaystyle \sum_{uv\in E}area(\Gamma_I),
\end{array}
\end{equation}
where $\Gamma_c$ and $\Gamma_I$ represent the circumscribed circle and the inscribed circle of the triangle $AOB$, while $R$ and $r$ denote their corresponding radii. The expressions for the perimeter and area of these circles are $perim(\Gamma) = 2\pi \rho$ and $area(\Gamma) = \pi \rho^2$, for any corresponding radius $\rho \in \{R, r\}$.

Furthermore, we recall the definitions of the Albertson index ($Alb$), the Tangent Sombor index ($Tan$) \cite{oz2025more}, the sigma irregularity index ($\sigma$), the Arithmetic-Geometric index ($AG$), and the Triangle Volume index ($SO_{11}$):
\begin{equation}
\label{Eq:Additional_indices}
\begin{array}{ll}
	Alb(G)=\displaystyle\sum_{uv\in E}(y-x)=\dfrac{1}{\sqrt{2}}\displaystyle \sum_{uv\in E}l_3, 
	& 
	Tan(G)=\displaystyle \sum_{uv\in E}\dfrac{y^2-x^2}{2xy}, \\[2.5em]
	\sigma(G)=\displaystyle\sum_{uv\in E}(y-x)^2,
    &
    AG(G)=\displaystyle\sum_{uv\in E}\frac{x+y}{2\sqrt{xy}}, \\[2.5em]
    SO_{11}(G)=\displaystyle\sum_{uv\in E}(y^2 x - x^2 y + y^3 - x^3). &
\end{array}
\end{equation}

\section{Linear Optimization Problem and the Standard Basis}
\label{sec:linear_opt}

Let $n$ be the order of a tree $T$ and $\Delta$ be its maximum degree. A tree that maximizes a given topological index $TI(T)$ over the class of all such trees is formally defined as an extremal tree. Let $n_i$ denote the number of vertices of degree $i$, and let $m_{i,j}$ denote the number of edges connecting vertices of degrees $i$ and $j$. The variables $n_i$ and $m_{i,j}$ are strictly related by the following system of identities, which hold for every tree $T$:

\begin{align}
n_1 + n_2 + \cdots + n_{\Delta} &= n, \label{eq:tree-sum-ni}\\
n_1 + 2n_2 + \cdots + \Delta n_{\Delta} &= 2n-2, \label{eq:tree-sum-deg}\\
m_{1,2} + m_{1,3} + \cdots + m_{1,\Delta} &= n_1, \label{eq:tree-ni-1}\\
m_{1,2} + 2m_{2,2} + m_{2,3} + \cdots + m_{2,\Delta} &= 2n_2, \label{eq:tree-ni-2}\\
m_{1,3} + m_{2,3} + 2m_{3,3} + \cdots + m_{3,\Delta} &= 3n_3, \label{eq:tree-ni-3}\\
\intertext{\centering$\vdots$}
m_{1,i}+m_{2,i}+\cdots+m_{i-1,i}+2m_{i,i}+m_{i,i+1}+\cdots+m_{i,\Delta}
&= i n_i, \label{eq:tree-ni-i}\\
\intertext{\centering$\vdots$}
m_{1,\Delta} + m_{2,\Delta} + m_{3,\Delta} + \cdots + 2m_{\Delta,\Delta}
&= \Delta n_\Delta. \label{eq:tree-ni-Delta}
\end{align}

From these structural identities, we can strictly derive two foundational constraints for our optimization model. First, summing the degree-pair variables directly recovers the fundamental tree property regarding the total number of edges, $\sum m_{i,j} = n - 1$. Second, dividing each degree equation \eqref{eq:tree-ni-i} by $i$ and subsequently summing the resulting system over all degrees $1 \le i \le \Delta$ perfectly reconstructs the vertex sum \eqref{eq:tree-sum-ni}, yielding the mathematically elegant identity $\sum \left(\frac{1}{i}+\frac{1}{j}\right)m_{i,j} = n$.

To systematically calculate any generalized vertex-degree-based topological index $TI(T)$, we formulate the objective function as a inear combination of these degree-pair multiplicities and their corresponding topological weights. We order the indices such that the first index is strictly less than the second, evaluating weight functions for $1 \le i \le j \le \Delta$. Thus, the problem of finding the extremal value of a topological index reduces to the following continuous linear program (LP):

\begin{align}
    \text{Maximize} \quad & TI(T) = \sum_{1 \le i \le j \le \Delta} f(i,j)\, m_{i,j} \label{eq:LP-objective-g} \\
    \text{subject to} \quad & \sum_{1 \le i \le j \le \Delta} m_{i,j} = n - 1, \label{eq:LP-constraint-sum-SO} \\
    & \sum_{1 \le i \le j \le \Delta} \Bigl(\frac{1}{i}+\frac{1}{j}\Bigr)\, m_{i,j} = n, \label{eq:LP-constraint-w-SO} \\
    & m_{i,j} \ge 0 \quad \text{for all } 1 \le i \le j \le \Delta. \nonumber
\end{align}

\subsection{Duality and the Gap Function Framework}

To establish the uniqueness of an extremal tree structure, we rely on the fundamental principles of Linear Programming (LP) duality. In combinatorial optimization, a specific graph topology corresponds to a basic feasible solution of the primal LP, where the active edges (the ``basis'') have strictly positive weights ($m_{i,j} > 0$), while all absent edges are non-basic ($m_{i,j} = 0$).

By associating dual variables $y_1$ and $y_2$ with constraints \eqref{eq:LP-constraint-sum-SO} and \eqref{eq:LP-constraint-w-SO}, we construct the dual problem. The dual constraints dictate that for every possible edge type $(i,j)$, the inequality $y_1 + \left(\frac{1}{i}+\frac{1}{j}\right)y_2 \ge f(i,j)$ must hold. 

According to the theorem of \textit{strict complementary slackness}, an optimal graph structure is absolutely unique if and only if:
\begin{enumerate}
    \item For every edge in the optimal basis, the corresponding dual constraint is binding (an exact equality).
    \item For every edge \textit{not} in the optimal basis, the corresponding dual constraint holds as a strict inequality.
\end{enumerate}
We formalize this strict inequality margin as the \textit{dual gap function} (or reduced cost), denoted $G(i,j)$. If $G(i,j) > 0$ for all non-basic edges, introducing any alternative edge would strictly penalize the objective function, thereby guaranteeing the structural uniqueness of the extremal tree.

A direct and powerful consequence of this duality framework is the analytical decomposition of the topological index into an exact continuous upper bound and a discrete penalty function. From the definition of the dual gap function, the topological weight of any edge $(i,j)$ can be algebraically expressed as:
\begin{equation}
f(i,j) = y_1 + \left( \frac{1}{i} + \frac{1}{j} \right) y_2 - G(i,j).
\end{equation}
By substituting this identity into the objective function \eqref{eq:LP-objective-g}, the total topological index expands to:
\begin{equation}
TI(T) = \sum_{1 \le i \le j \le \Delta} \left[ y_1 + \left( \frac{1}{i} + \frac{1}{j} \right) y_2 - G(i,j) \right] m_{i,j}.
\end{equation}
Distributing the summation over the terms and factoring out the dual variables yields:
\begin{equation}
TI(T) = y_1 \sum_{1 \le i \le j \le \Delta} m_{i,j} + y_2 \sum_{1 \le i \le j \le \Delta} \left( \frac{1}{i} + \frac{1}{j} \right) m_{i,j} - \sum_{1 \le i \le j \le \Delta} G(i,j)\, m_{i,j}.
\end{equation}
By substituting the primal equality constraints \eqref{eq:LP-constraint-sum-SO} and \eqref{eq:LP-constraint-w-SO} directly into the first two summations, the dependence on individual degree-pair multiplicities strictly collapses, leaving only the penalty sum. Grouping the terms by the tree order $n$ produces a closed-form universal linear equation:
\begin{equation}
\label{eq:general_ti_decomposition}
TI(T) = (y_1 + y_2)n - y_1 - \sum_{(i,j) \notin \mathcal{B}} G(i,j)\, m_{i,j}.
\end{equation}
Notice that the penalty sum is strictly evaluated over non-basic edges, since strict complementary slackness guarantees $G(i,j) = 0$ for all $(i,j) \in \mathcal{B}$. 

This rigorous algebraic derivation proves that for any chosen optimal basis $\mathcal{B}$, the continuous LP maximum equates to $TI_{LP}^* = A \cdot n + B$, where the precise analytical coefficients are uniquely determined by the dual variables:
\begin{equation}
A = y_1 + y_2, \qquad B = -y_1.
\end{equation}
Consequently, the problem of maximizing any topological index over a discrete graph is mathematically equivalent to minimizing the discrete dual penalty function $\mathcal{P}(T) = \sum G(i,j)\, m_{i,j}$, establishing the foundational optimization mechanism utilized throughout this paper.

To leverage calculus within this optimization framework, we henceforth treat the discrete vertex degrees $x$ and $y$ as continuous real variables on the interval $[1, \Delta]$. This continuous extension allows us to evaluate the analytical behavior and monotonicity of the topological weight functions using partial derivatives, naturally restricting the final structural optimization back to discrete integer evaluations.

\subsection{Optimality of the Standard $\{(1, \Delta), (2, \Delta)\}$ Basis}

Through this duality framework, we formulate a general condition for an index function $f(i,j)$ to attain its unique optimal value on the standard extremal configuration strictly restricted to the edge basis $\{(1, \Delta), (2, \Delta)\}$.

\begin{lemma}[Optimality Criterion]
\label{lem:general-f-condition}
Let $\Delta \ge 3$ and $n \ge \Delta+2$ be integers. The linear program consisting of maximizing \eqref{eq:LP-objective-g} subject to \eqref{eq:LP-constraint-sum-SO} and \eqref{eq:LP-constraint-w-SO} admits a unique optimal basic feasible solution strictly comprising the edges $m_{1, \Delta}$ and $m_{2, \Delta}$, where
\[
m_{1, \Delta} \;=\; \frac{(\Delta-2)n + (\Delta+2)}{\Delta}, 
\qquad
m_{2, \Delta} \;=\; \frac{2(n-\Delta-1)}{\Delta},
\]
if and only if for all valid non-basic pairs $(i,j) \notin \{(1, \Delta), (2, \Delta)\}$, the dual gap function satisfies $G(i,j) > 0$, where:
\begin{equation}
\label{eq:general-gap-function}
G(i,j) = f(1, \Delta) - f(i,j) - y_2 \left( 1 + \frac{1}{\Delta} - \frac{1}{i} - \frac{1}{j} \right),
\end{equation}
and the dual penalty coefficient evaluates to $y_2 = 2 \big( f(1, \Delta) - f(2, \Delta) \big)$.
\end{lemma}

\begin{proof}
For the basic variables $m_{1, \Delta} > 0$ and $m_{2, \Delta} > 0$, strict complementary slackness dictates that their corresponding dual constraints are exact equalities:
\begin{align}
y_1 + \left(1 + \frac{1}{\Delta}\right)y_2 &= f(1, \Delta), \label{eq:dual_eq1}\\
y_1 + \left(\frac{1}{2} + \frac{1}{\Delta}\right)y_2 &= f(2, \Delta). \label{eq:dual_eq2}
\end{align}
Subtracting \eqref{eq:dual_eq2} from \eqref{eq:dual_eq1} yields $\frac{1}{2}y_2 = f(1, \Delta) - f(2, \Delta)$, explicitly evaluating the dual penalty coefficient as $y_2 = 2 \big( f(1, \Delta) - f(2, \Delta) \big)$. 

By isolating $y_1$ from \eqref{eq:dual_eq1}, we obtain $y_1 = f(1, \Delta) - \left(1 + \frac{1}{\Delta}\right)y_2$. For uniqueness, dual feasibility strictly requires the dual gap function to be positive for all non-basic edges:
\[
G(i,j) = y_1 + \left(\frac{1}{i}+\frac{1}{j}\right)y_2 - f(i,j) > 0.
\]
Substituting the derived expression for $y_1$ directly yields exactly the algebraic formulation of $G(i,j)$ in \eqref{eq:general-gap-function}, thus completing the proof.
\end{proof}

To bypass the complex task of analyzing a two-variable discrete optimization problem, we introduce a dimensional reduction technique. This lemma proves that the worst-case behavior of the gap function always occurs at the boundary where the larger degree reaches its maximum possible value, $y = \Delta$.

\begin{lemma}[Dimensionality Reduction]
\label{lem:reduction}
Let $f(x,y)$ be a continuous, differentiable extension of a topological weight function evaluated for $1 \le x < y \le \Delta$. If $f$ is strictly increasing with respect to its second variable, i.e., $\frac{\partial f}{\partial y} > 0$, and if $f(1, \Delta) > f(2, \Delta)$, then for any fixed $x$, the gap function $G(x,y)$ is strictly decreasing with respect to $y$. Consequently, for any valid discrete degrees $3 \le i < j \le \Delta$, the inequality $G(i,j) \ge G(i, \Delta)$ holds.
\end{lemma}

\begin{proof}
By definition, the gap function extended to continuous variables $x$ and $y$ is given by:
\begin{equation}
G(x,y) = f(1,\Delta) - f(x,y) - y_2 \left( 1 + \frac{1}{\Delta} - \frac{1}{x} - \frac{1}{y} \right),
\end{equation}
with $y_2 > 0$ since $f(1, \Delta) > f(2, \Delta)$. 

Taking the partial derivative of $G(x,y)$ with respect to the second variable $y$, we obtain:
\begin{equation}
\frac{\partial G}{\partial y} = -\frac{\partial f(x,y)}{\partial y} - \frac{y_2}{y^2}.
\end{equation}
Since we assumed that $\frac{\partial f}{\partial y} > 0$ and knowing that $y_2 > 0$, both terms on the right side of the equation are strictly negative. Therefore, $\frac{\partial G}{\partial y} < 0$ for all valid $x, y$. 

This implies that for a fixed $x$, the function $G(x,y)$ monotonically decreases as $y$ increases. Since the maximum possible degree in the graph is bounded by $\Delta$, the function attains its global minimum with respect to $y$ at the upper boundary $y = \Delta$. Thus, restricted to discrete vertex degrees, we conclude that $G(i,j) \ge G(i, \Delta)$ for all $1 \le i < j \le \Delta$.
\end{proof}

\begin{theorem}[Main Optimality Theorem]
\label{thm:main_1D}
Let $T$ be a tree of order $n$ and maximum degree $\Delta \ge 3$. Let $f$ be a topological weight function possessing a continuous and differentiable extension $f(x,y)$ for $1 \le x < y \le \Delta$. If $f(x,y)$ is strictly increasing with respect to its second variable (i.e., $\frac{\partial f}{\partial y} > 0$) and satisfies $f(1, \Delta) > f(2, \Delta)$, define the single-variable test function $h(x)$ on the interval $[3, \Delta]$ as:
\begin{equation}
\label{eq:test_hx}
h(x) = f(1,\Delta) - f(x,\Delta) - y_2 \left( 1 - \frac{1}{x} \right),
\end{equation}
where $y_2 = 2\big(f(1,\Delta) - f(2,\Delta)\big)$. If $h(x) > 0$ for all $x \in [3, \Delta]$, then the tree configuration maximizing the corresponding topological index is uniquely determined.
\end{theorem}

\begin{proof}
According to Lemma \ref{lem:general-f-condition}, the optimal configuration is uniquely determined if the gap function satisfies $G(i,j) > 0$ for all internal nodes with degrees $3 \le i < j \le \Delta$. 

By applying Lemma \ref{lem:reduction}, we can establish a lower bound for $G(i,j)$ by projecting it onto the upper boundary $j = \Delta$:
\begin{equation}
G(i,j) \ge G(i, \Delta) = f(1,\Delta) - f(i,\Delta) - y_2 \left( 1 + \frac{1}{\Delta} - \frac{1}{i} - \frac{1}{\Delta} \right).
\end{equation}
Simplifying the expression by cancelling the $1/\Delta$ terms yields exactly our test function evaluated at $i$:
\begin{equation}
G(i,j) \ge h(i).
\end{equation}
Since, by the premise of the theorem, $h(x) > 0$ for all real values $x \in [3, \Delta]$, it trivially holds for all valid discrete degree values $i \in \{3, \dots, \Delta-1\}$. Consequently, $G(i,j) > 0$, fulfilling the necessary condition and concluding the proof.
\end{proof}

\section{Structural Realization and Applications}
\label{sec:applications}

To bridge the continuous linear programming framework from Section \ref{sec:linear_opt} with discrete graph theory, we must formally characterize the exact parameters under which a tree realizes a basic feasible solution. We establish a universal generative structural class for any bipartite basis parameterized by $\{(1, \Delta), (k, \Delta)\}$. By defining this universally, we simultaneously provide the exact topological construction for the standard basis ($k=2$) utilized in this section, the alternative basis ($k=\Delta$) evaluated in Section \ref{sec:alternative_opt}, and the degree-dependent basis derived in Section \ref{sec:degree_dependent_basis_so11}.

A tree restricted to this bipartite basis contains exclusively leaves ($V_1$), degree-$k$ vertices ($V_k$), and maximum degree vertices ($V_\Delta$). Because every edge must connect to a vertex in $V_\Delta$, the incidence identity \eqref{eq:tree-ni-Delta} simplifies to $m_{1, \Delta} + m_{k, \Delta} = \Delta n_\Delta$. Since the total number of edges is $n - 1$, this directly isolates $n_\Delta$:
\begin{equation}
\label{eq:basis_k_ndelta}
n_\Delta = \frac{n - 1}{\Delta}.
\end{equation}
Similarly, the incidence identities for $V_1$ and $V_k$ evaluate to $m_{1, \Delta} = n_1$ and $m_{k, \Delta} = k n_k$. Summing these equations yields $n_1 + k n_k = n - 1$. Subtracting the vertex sum identity ($n_1 + n_k + n_\Delta = n$) from this edge sum produces the exact linear relation:
\begin{equation}
\label{eq:basis_k_relation}
n_\Delta - (k-1)n_k = 1.
\end{equation}
Substituting Eq. \eqref{eq:basis_k_ndelta} into Eq. \eqref{eq:basis_k_relation} isolates $n_k$:
\begin{equation}
\label{eq:basis_k_nk}
n_k = \frac{n - \Delta - 1}{\Delta(k-1)}.
\end{equation}
For this topology to exist, $n_k$ must be a non-negative integer. This establishes the strict Diophantine congruence:
\begin{equation}
\label{eq:basis_k_modulo}
n \equiv \Delta + 1 \pmod{\Delta(k-1)}.
\end{equation}
By parameterizing the degree-$k$ vertex count as $n_k = q$ for some integer $q \ge 0$, the exact tree order is defined as $n = q\Delta(k-1) + \Delta + 1$. Substituting $q$ and $n$ back into the structural identities yields the complete integer profile:
\begin{align*}
n_\Delta &= q(k-1) + 1, \\
m_{k, \Delta} &= k q, \\
m_{1, \Delta} = n_1 &= q(\Delta k - \Delta - k) + \Delta.
\end{align*}

To formalize the topological existence of graphs matching these parameters, we define the corresponding universal tree class.

\begin{definition}
Let $2 \le k \le \Delta$ and $q \ge 0$ be integers. We define $\mathcal{T}(k, \Delta, q)$ as the family of all trees constructed by the following procedure: Let $H$ be an arbitrary tree with vertex bipartition $V_k \cup V_\Delta$ such that $|V_k| = q$ and $|V_\Delta| = q(k-1) + 1$. Every vertex in $V_k$ must have degree exactly $k$ in $H$. For each vertex $v \in V_\Delta$, attach exactly $\Delta - d_H(v)$ pendant vertices to it, where $d_H(v)$ is its degree in the subtree $H$.
\end{definition}

By construction, all internal edges of a tree in $\mathcal{T}(k, \Delta, q)$ connect $V_k$ and $V_\Delta$, yielding exactly $m_{k, \Delta} = k q$. The total number of attached pendant edges is:
\begin{equation*}
m_{1, \Delta} = \sum_{v \in V_\Delta} \big(\Delta - d_H(v)\big) = \Delta |V_\Delta| - \sum_{v \in V_\Delta} d_H(v).
\end{equation*}
Because all edges in the subtree $H$ connect $V_k$ and $V_\Delta$, the sum of degrees in $V_\Delta$ equals the total number of edges in $H$, which is $kq$. Substituting $|V_\Delta|$ evaluates the exact number of leaves as $\Delta(q(k-1) + 1) - kq = q(\Delta k - \Delta - k) + \Delta$, identically matching the expression for $n_1$.

Conversely, any tree restricted to the basis $\{(1, \Delta), (k, \Delta)\}$ has all internal edges connecting $V_k$ to $V_\Delta$. Pruning all leaves yields a subtree $H$ where the $q$ vertices in $V_k$ retain their exact degree $k$, and the $q(k-1)+1$ vertices in $V_\Delta$ have some degree $d_H(v) \le \Delta$. This confirms the graph was constructed by attaching leaves to $H$, proving that $\mathcal{T}(k, \Delta, q)$ uniquely characterizes all valid tree configurations for $n \equiv \Delta + 1 \pmod{\Delta(k-1)}$.

For the standard basis evaluated in Theorem \ref{thm:main_1D}, we set $k=2$. The Diophantine congruence \eqref{eq:basis_k_modulo} simplifies to $n \equiv \Delta + 1 \pmod{\Delta}$, which is equivalent to $n \equiv 1 \pmod{\Delta}$. An illustrative topological realization of this standard basis is provided in Figure \ref{fig:optimal_standard}.

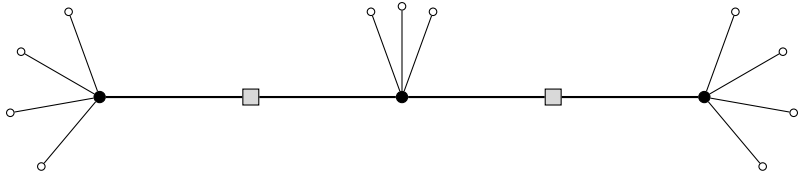
\begin{figure}[htpb]
    \centering
    \begin{tikzpicture}[
        v2/.style={rectangle, draw=black, fill=gray!30, inner sep=2.5pt, minimum size=6pt},
        v5/.style={circle, draw=black, fill=black, inner sep=1.5pt},
        leaf/.style={circle, draw=black, fill=white, inner sep=1pt},
        edge/.style={thick}
    ]
    
    \node[v2] (K1) at (0, 0) {};
    \node[v2] (K2) at (4, 0) {};
    
    \node[v5] (D1) at (-2, 0) {};
    \node[v5] (D2) at (2, 0) {};
    \node[v5] (D3) at (6, 0) {};
    
    \draw[edge] (K1) -- (D1);
    \draw[edge] (K1) -- (D2);
    \draw[edge] (K2) -- (D2);
    \draw[edge] (K2) -- (D3);
    
    \foreach \angle/\i in {110/1, 150/2, 190/3, 230/4} {
        \node[leaf] (L1\i) at ([shift={(\angle:1.2)}]D1) {};
        \draw (D1) -- (L1\i);
    }
    
    \foreach \angle/\i in {70/1, 90/2, 110/3} {
        \node[leaf] (L2\i) at ([shift={(\angle:1.2)}]D2) {};
        \draw (D2) -- (L2\i);
    }
    
    \foreach \angle/\i in {310/1, 350/2, 30/3, 70/4} {
        \node[leaf] (L3\i) at ([shift={(\angle:1.2)}]D3) {};
        \draw (D3) -- (L3\i);
    }
    \end{tikzpicture}
\caption{An extremal tree belonging to the universal class $\mathcal{T}(2, \Delta, q)$ for the standard basis $\{(1, \Delta), (2, \Delta)\}$. Illustrated is the configuration for $\Delta=5$ and $q=2$, yielding order $n=16$. The two degree-2 vertices (squares) exclusively connect the three maximum degree vertices (solid circles).}
\label{fig:optimal_standard}
\end{figure}

Having established the topological existence of these extremal graphs, we now systematically apply Theorem \ref{thm:main_1D} to maximize several prominent topological indices. Consequently, in the following corollaries, whenever the necessary conditions are met and $n \equiv 1 \pmod{\Delta}$, the optimal tree uniquely maps to the class $\mathcal{T}(2, \Delta, q)$.

\begin{corollary}[Maximum Tangent Index]
Let $T$ be a tree of order $n$ with maximum degree $\Delta \ge 3$. The maximum value of the Tangent index is uniquely achieved by the extremal tree configuration restricted strictly to the edge basis $\{(1, \Delta), (2, \Delta)\}$.
\end{corollary}

\begin{proof}
For the Tangent index, the weight function under the convention $x < y$ is defined as $f(x,y) = \frac{y^2 - x^2}{2xy} = \frac{y}{2x} - \frac{x}{2y}$. We first verify the monotonicity condition of Theorem \ref{thm:main_1D}:
$$
\frac{\partial f}{\partial y} = \frac{1}{2x} + \frac{x}{2y^2} > 0.
$$
Next, we determine the penalty coefficient $y_2$ and simultaneously verify the second requirement $f(1, \Delta) > f(2, \Delta)$:
\begin{equation*}
y_2 = 2\big(f(1,\Delta) - f(2,\Delta)\big) = 2 \left( \frac{\Delta^2 - 1}{2\Delta} - \frac{\Delta^2 - 4}{4\Delta} \right) = \frac{\Delta^2 + 2}{2\Delta}.
\end{equation*}
Since $\Delta \ge 3$, it strictly follows that $y_2 > 0$, thereby confirming $f(1, \Delta) > f(2, \Delta)$. With both conditions of Theorem \ref{thm:main_1D} fulfilled, the optimality reduces to evaluating the single-variable test function defined in \eqref{eq:test_hx}.

Substituting $f(x, \Delta)$ and $y_2$ back into \eqref{eq:test_hx}, we obtain the expression for $h(x)$:
\begin{equation}
h(x) = \frac{\Delta^2 - 1}{2\Delta} - \frac{\Delta^2 - x^2}{2\Delta x} - \frac{\Delta^2 + 2}{2\Delta} \left( \frac{x-1}{x} \right).
\end{equation}
Simplifying this expression algebraically yields a concise closed-form equation:
\begin{equation}
h(x) = \frac{(x-1)(x-2)}{2\Delta x}.
\end{equation}
Taking the first derivative yields $h'(x) = \frac{1}{2\Delta}\left(1 - \frac{2}{x^2}\right)$. For all $x \ge 3$, it holds that $x^2 > 2$, which implies $h'(x) > 0$. Since $h(1)=0$ and $h(2)=0$, the function is strictly positive for all $x \in [3, \Delta]$. According to Theorem \ref{thm:main_1D}, this completes the proof.
\end{proof}

\begin{corollary}[$SO_1$ Index]
\label{cor:so1}
Let $T$ be a tree of order $n$ with maximum degree $\Delta \ge 3$. The maximum value of the $SO_1$ index is uniquely achieved by the extremal tree configuration restricted strictly to the edge basis $\{(1, \Delta), (2, \Delta)\}$.
\end{corollary}

\begin{proof}
The topological weight function characterizing the $SO_1$ index is given by $f(x,y) = \frac{1}{2}(y^2 - x^2)$ for $x < y$. We first establish the monotonicity prerequisite of Theorem \ref{thm:main_1D} by verifying that $\frac{\partial f}{\partial y} = y > 0$.

Next, we evaluate the penalty coefficient $y_2$ to verify the condition $f(1, \Delta) > f(2, \Delta)$:
\begin{equation*}
y_2 = 2\big(f(1,\Delta) - f(2,\Delta)\big) = (\Delta^2 - 1) - (\Delta^2 - 4) = 3.
\end{equation*}
Since $y_2 = 3 > 0$, the initial conditions of Theorem \ref{thm:main_1D} are fully satisfied. Substituting $f(x, \Delta)$ and $y_2$ into the 1D test function \eqref{eq:test_hx}, we obtain:
\begin{equation}
h(x) = \frac{\Delta^2 - 1}{2} - \frac{\Delta^2 - x^2}{2} - 3 \left( \frac{x-1}{x} \right) = \frac{x^2 - 1}{2} - \frac{3(x-1)}{x}.
\end{equation}
By factoring out $(x-1)$, this algebraic expression simplifies directly into a concise closed-form equation:
\begin{equation}
h(x) = \frac{(x-1)(x-2)(x+3)}{2x}.
\end{equation}
Since all factors of $h(x)$ are strictly positive for $x \ge 3$, we have $h(x) > 0$ for all $x \in [3, \Delta]$. By Theorem \ref{thm:main_1D}, the extremal tree uniquely maximizes the $SO_{11}$ index.
\end{proof}

While $SO_1$ evaluates geometric area, the $SO_2$ index evaluates the angular metric of the geometric representation. We demonstrate that the identical standard basis remains uniquely optimal.

\begin{corollary}[$SO_2$ Index]
\label{cor:so2}
Let $T$ be a tree of order $n$ with maximum degree $\Delta \ge 3$. The maximum value of the $SO_2$ index is uniquely achieved by the extremal tree configuration restricted strictly to the edge basis $\{(1, \Delta), (2, \Delta)\}$.
\end{corollary}

\begin{proof}
The topological weight function characterizing the $SO_2$ index is given by $f(x,y) = \frac{y^2 - x^2}{x^2 + y^2}$ for $x < y$. We first establish the monotonicity prerequisite of Theorem \ref{thm:main_1D} by evaluating the partial derivative with respect to $y$:
\begin{equation*}
\frac{\partial f}{\partial y} = \frac{2y(x^2 + y^2) - 2y(y^2 - x^2)}{(x^2 + y^2)^2} = \frac{4x^2y}{(x^2 + y^2)^2}.
\end{equation*}
Since $x, y \ge 1$, it strictly follows that $\frac{\partial f}{\partial y} > 0$. 

Next, we evaluate the penalty coefficient $y_2$ to confirm the condition $f(1, \Delta) > f(2, \Delta)$:
\begin{align*}
y_2 &= 2\left( \frac{\Delta^2 - 1}{\Delta^2 + 1} - \frac{\Delta^2 - 4}{\Delta^2 + 4} \right) \\
&= 2 \left( \frac{(\Delta^4 + 3\Delta^2 - 4) - (\Delta^4 - 3\Delta^2 - 4)}{(\Delta^2 + 1)(\Delta^2 + 4)} \right) = \frac{12\Delta^2}{(\Delta^2 + 1)(\Delta^2 + 4)}.
\end{align*}
Given $\Delta \ge 3$, $y_2 > 0$, completely satisfying the initial conditions of Theorem \ref{thm:main_1D}. Substituting $f(x, \Delta)$ and $y_2$ into the 1D test function \eqref{eq:test_hx}, we obtain:
\begin{equation}
h(x) = \left( \frac{\Delta^2 - 1}{\Delta^2 + 1} - \frac{\Delta^2 - x^2}{\Delta^2 + x^2} \right) - \frac{12\Delta^2}{(\Delta^2 + 1)(\Delta^2 + 4)} \left( \frac{x-1}{x} \right).
\end{equation}
Algebraically simplifying the first parenthetical term yields $\frac{2\Delta^2(x^2 - 1)}{(\Delta^2 + 1)(\Delta^2 + x^2)}$. By factoring out the common multiplier $\frac{2\Delta^2(x-1)}{\Delta^2 + 1}$, the function can be compactly rewritten as:
\begin{equation}
h(x) = \frac{2\Delta^2(x-1)}{\Delta^2 + 1} \left[ \frac{x+1}{\Delta^2 + x^2} - \frac{6}{x(\Delta^2 + 4)} \right].
\end{equation}
For $h(x) > 0$, we must show that the bracketed expression is positive, which is algebraically equivalent to proving that the cross-multiplied numerator $\mathcal{N}(x, \Delta) = x(x+1)(\Delta^2 + 4) - 6(\Delta^2 + x^2) > 0$. Expanding and regrouping this polynomial gives:
\begin{equation}
\mathcal{N}(x, \Delta) = (x^2 + x - 6)\Delta^2 - 2x^2 + 4x.
\end{equation}
For all valid interior degrees $x \ge 3$, the quadratic factor $(x^2 + x - 6) = (x-2)(x+3)$ is strictly positive. Since the maximum degree satisfies $\Delta \ge x$, we can establish a strict lower bound by substituting $\Delta^2 \ge x^2$:
\begin{equation}
\mathcal{N}(x, \Delta) \ge x^2(x^2 + x - 6) - 2x^2 + 4x = x^4 + x^3 - 8x^2 + 4x.
\end{equation}
This lower bound polynomial factors elegantly into $x(x-2)(x^2 + 3x - 2)$. For $x \ge 3$, every individual factor is strictly positive, definitively ensuring $\mathcal{N}(x, \Delta) > 0$. Consequently, $h(x) > 0$ for all $x \in [3, \Delta]$. By Theorem \ref{thm:main_1D}, the extremal tree configuration uniquely maximizes the $SO_2$ index.
\end{proof}

\begin{corollary}[Arithmetic-Geometric Index]
\label{cor:ag}
Let $T$ be a tree of order $n$ with maximum degree $\Delta \ge 9$. The maximum value of the Arithmetic-Geometric ($AG$) index is uniquely achieved by the extremal tree configuration restricted strictly to the edge basis $\{(1, \Delta), (2, \Delta)\}$.
\end{corollary}

\begin{proof}
The weight function for the $AG$ index under the convention $x < y$ is defined as $f(x,y) = \frac{x+y}{2\sqrt{xy}}$. We first verify the monotonicity condition of Theorem \ref{thm:main_1D} by taking the partial derivative with respect to $y$:
\begin{equation*}
\frac{\partial f}{\partial y} = \frac{y-x}{4\sqrt{x}y^{3/2}}.
\end{equation*}
Since $y > x \ge 1$, it strictly follows that $\frac{\partial f}{\partial y} > 0$. 

Next, we evaluate the penalty coefficient $y_2$ to confirm the condition $f(1, \Delta) > f(2, \Delta)$:
\begin{equation*}
y_2 = 2\big(f(1,\Delta) - f(2,\Delta)\big) = 2 \left( \frac{\Delta + 1}{2\sqrt{\Delta}} - \frac{\Delta + 2}{2\sqrt{2\Delta}} \right) = \frac{\sqrt{2}(\Delta + 1) - (\Delta + 2)}{\sqrt{2\Delta}}.
\end{equation*}
To guarantee $y_2 > 0$, the linear numerator of this expression must be strictly positive. Rearranging the terms into a linear function of $\Delta$ yields:
\begin{equation*}
(\sqrt{2}-1)\Delta - (2-\sqrt{2}) > 0 \implies \Delta > \frac{2-\sqrt{2}}{\sqrt{2}-1}  = \sqrt{2}.
\end{equation*}
Since our maximum degree satisfies $\Delta \ge 9 > \sqrt{2}$, this condition is trivially fulfilled. With the initial conditions of Theorem \ref{thm:main_1D} established, we construct the single-variable test function $h(x)$ from \eqref{eq:test_hx}:
\begin{equation}
h(x) = \frac{\Delta + 1}{2\sqrt{\Delta}} - \frac{x + \Delta}{2\sqrt{x\Delta}} - \frac{(\sqrt{2}-1)\Delta - (2-\sqrt{2})}{\sqrt{2\Delta}} \left( \frac{x-1}{x} \right).
\end{equation}
By factoring out the common positive denominator $2x\sqrt{\Delta}$, we can rewrite the test function as $h(x) = \frac{\mathcal{N}(x, \Delta)}{2x\sqrt{\Delta}}$, where the algebraic numerator is:
\begin{equation}
\mathcal{N}(x, \Delta) = x(\Delta+1) - \sqrt{x}(x+\Delta) - \sqrt{2}(x-1)\big[(\sqrt{2}-1)\Delta - (2-\sqrt{2})\big].
\end{equation}
To prove $h(x) > 0$, it is sufficient to prove $\mathcal{N}(x, \Delta) > 0$. Expanding and grouping the terms into a linear function of the parameter $\Delta$, we obtain:
\begin{equation}
\mathcal{N}(x, \Delta) = \Delta \big[ (\sqrt{2}-1)x - \sqrt{x} + 2 - \sqrt{2} \big] - x\sqrt{x} + (2\sqrt{2}-1)x - (2\sqrt{2}-2).
\end{equation}
Let $C_1(x) = (\sqrt{2}-1)x - \sqrt{x} + 2 - \sqrt{2}$ represent the slope of $\mathcal{N}(x, \Delta)$ with respect to $\Delta$. By treating $C_1(x)$ as a quadratic function of the variable $t = \sqrt{x}$, we calculate its discriminant:
\begin{equation}
D = (-1)^2 - 4(\sqrt{2}-1)(2-\sqrt{2}) = 17 - 12\sqrt{2} = (3-2\sqrt{2})^2.
\end{equation}
The roots of this quadratic are exactly $t_1 = 1$ and $t_2 = \sqrt{2}$. With a positive leading coefficient, the parabola is strictly positive for all $t > \sqrt{2}$. Since valid internal degrees require $x \ge 3$, we have $t = \sqrt{x} \ge \sqrt{3} > \sqrt{2}$, definitively proving $C_1(x) > 0$. 

Because the slope is positive, $\mathcal{N}(x, \Delta)$ is a strictly increasing function of $\Delta$. Therefore, for any valid tree with $\Delta \ge 9$, the numerator is bounded below by its evaluation at the minimum boundary $\Delta = 9$:
\begin{equation}
\mathcal{N}(x, \Delta) \ge \mathcal{N}(x, 9) = -x\sqrt{x} + (11\sqrt{2}-10)x - 9\sqrt{x} + (20-11\sqrt{2}).
\end{equation}
Introducing the substitution $t = \sqrt{x}$, the minorant becomes a cubic polynomial $P(t) = -t^3 + (11\sqrt{2}-10)t^2 - 9t + (20-11\sqrt{2})$. Since $x \in [3, 9]$, we analyze $P(t)$ over the domain $t \in [\sqrt{3}, 3]$. Taking the first derivative yields $P'(t) = -3t^2 + 2(11\sqrt{2}-10)t - 9$. The roots of $P'(t) = 0$ dictate the local extrema. Solving this quadratic equation reveals a local minimum at $t_1 \approx 1.196$ (which lies outside our domain) and a local maximum at $t_2 \approx 2.508$ (which lies inside our domain). 

This implies that on the interval $[\sqrt{3}, 3]$, the function $P(t)$ monotonically increases to its local maximum at $t_2$ and then monotonically decreases. Due to this, the absolute minimum of $P(t)$ on this interval must occur at one of the boundary points. Evaluating $P(t)$ at these endpoints gives exactly:
\begin{align}
P(\sqrt{3}) &= 22\sqrt{2} - 12\sqrt{3} - 10 \approx 0.328 > 0, \\
P(3) &= 88\sqrt{2} - 124 \approx 0.451 > 0.
\end{align}
Since the polynomial evaluates to strictly positive values at both bounds, $P(t) > 0$ for the entire interior of the interval. Consequently, $\mathcal{N}(x, \Delta) > 0$, which ensures $h(x) > 0$ for all valid $x \in [3, \Delta]$. According to Theorem \ref{thm:main_1D}, the extremal tree configuration uniquely maximizes the Arithmetic-Geometric index.
\end{proof}

\begin{remark}
The bound $\Delta \ge 9$ in Corollary \ref{cor:ag} is sharp. For instance, if $\Delta = 8$, evaluating the numerator $\mathcal{N}(x, 8)$ at the upper bound $x = 8$ yields:
\begin{equation*}
\mathcal{N}(8, 8) = -16\sqrt{2} + 8(10\sqrt{2}-9) - 16\sqrt{2} + (18-10\sqrt{2}) = 38\sqrt{2} - 54.
\end{equation*}
Since $(38\sqrt{2})^2 = 2888 < 54^2 = 2916$, it follows that $38\sqrt{2} < 54$, implying $\mathcal{N}(8, 8) < 0$. Consequently, $h(x)$ fails to remain globally positive on $[3, 8]$, confirming that Theorem \ref{thm:main_1D} is not directly applicable for $\Delta \le 8$.
\end{remark}

For the Triangle Volume index ($SO_{11}$), the weight function $f(x,y) = y^3 + y^2x - x^2y - x^3$ fails to satisfy the preconditions of Theorem \ref{thm:main_1D}. Specifically, the difference $f(1, \Delta) - f(2, \Delta) = -\Delta^2 + 3\Delta + 7$ is strictly positive only for $\Delta \in \{3, 4\}$. For $\Delta \ge 5$, this difference is strictly negative, invalidating the dimensionality reduction. Therefore, to prove the optimality of the extremal tree configuration for the $SO_{11}$ index on the interval $3 \le \Delta \le 8$, we bypass Theorem \ref{thm:main_1D} entirely and directly apply the general optimality criterion formulated in Lemma \ref{lem:general-f-condition}.

\begin{lemma}[Positivity of the Gap Function for $SO_{11}$]
\label{lem:so11_gap}
Let $G(x,y)$ be the continuous extension of the gap function for the $SO_{11}$ index:
\begin{equation*}
G(x,y) = f(1, \Delta) - f(x,y) - y_2(\Delta) \left( 1 + \frac{1}{\Delta} - \frac{1}{x} - \frac{1}{y} \right),
\end{equation*}
where $f(x,y) = y^3 + y^2 x - x^2 y - x^3$ and $y_2(\Delta) = 2(-\Delta^2 + 3\Delta + 7)$. For any integer maximum degree $3 \le \Delta \le 8$, the strict inequality $G(x,y) > 0$ holds for all valid integer pairs $1 \le x < y \le \Delta$, excluding the basic variables $(1, \Delta)$ and $(2, \Delta)$.
\end{lemma}

\begin{proof}
To definitively prove global positivity $G(x,y) > 0$, we first analyze the dual penalty coefficient $y_2(\Delta)$. By definition, it evaluates to:
\begin{equation*}
y_2(\Delta) = 2\big(f(1,\Delta) - f(2,\Delta)\big) = 2(-\Delta^2 + 3\Delta + 7).
\end{equation*}
Taking the derivative $y_2'(\Delta) = -4\Delta + 6$, we observe that $y_2(\Delta)$ is strictly decreasing on the interval $[3, 8]$. Consequently, it attains its maximum at $\Delta = 3$ with $y_2(3) = 14$, and its absolute minimum at $\Delta = 8$ with $y_2(8) = -66$.

With these bounds established, we partition the valid domain into three strictly defined, mutually exclusive cases based on the value of $x$.

\textbf{Case 1:} $x = 1$ and $y \in [2, \Delta-1]$.
By the definition of the dual constraints, $G(1, \Delta) = 0$. We express $G(1,y)$ directly as:
\begin{equation*}
G(1,y) = f(1,\Delta) - f(1,y) + y_2(\Delta) \left( \frac{1}{y} - \frac{1}{\Delta} \right).
\end{equation*}
Substituting $f(1,y) = y^3 + y^2 - y - 1$ and factoring the difference of polynomials yields:
\begin{equation*}
f(1,\Delta) - f(1,y) = (\Delta - y)(\Delta^2 + \Delta y + y^2 + \Delta + y - 1).
\end{equation*}
We rewrite the fractional term as $y_2(\Delta) \frac{\Delta - y}{y\Delta}$ and factor out the common term $(\Delta - y)$:
\begin{equation*}
G(1,y) = (\Delta - y) \left[ \Delta^2 + \Delta y + y^2 + \Delta + y - 1 + \frac{y_2(\Delta)}{y\Delta} \right].
\end{equation*}
Since $y \le \Delta - 1$, the first factor is strictly positive $(\Delta - y \ge 1)$. Given $\Delta \ge 3$ and $y \ge 2$, the polynomial within the bracket is strictly bounded below by its evaluation at the minimal bounds:
\begin{equation*}
\Delta^2 + \Delta y + y^2 + \Delta + y - 1 \ge 3^2 + 3(2) + 2^2 + 3 + 2 - 1 = 23.
\end{equation*}
Next, we determine the lower bound for the fractional term $\frac{y_2(\Delta)}{y\Delta}$. Isolating the dependence on $\Delta$ yields:
\begin{equation*}
\frac{y_2(\Delta)}{\Delta} = \frac{-2\Delta^2 + 6\Delta + 14}{\Delta} = -2\Delta + 6 + \frac{14}{\Delta}.
\end{equation*}
Let $H(\Delta) = -2\Delta + 6 + \frac{14}{\Delta}$. Taking the derivative $H'(\Delta) = -2 - \frac{14}{\Delta^2} < 0$ dictates that $H(\Delta)$ is strictly decreasing on $[3, 8]$. It attains its absolute minimum at the upper boundary $\Delta = 8$:
\begin{equation*}
H(\Delta) \ge H(8) = -16 + 6 + \frac{14}{8} = -8.25.
\end{equation*}
Because this minimum is negative, dividing by the positive independent variable $y$ yields the most negative possible fraction when $y$ is minimized ($y = 2$). Therefore:
\begin{equation*}
\frac{y_2(\Delta)}{y\Delta} = \frac{1}{y} H(\Delta) \ge \frac{1}{2}(-8.25) = -4.125.
\end{equation*}
Summing these bounds, the bracketed expression is strictly greater than $23 - 4.125 = 18.875 > 0$. Being the product of two strictly positive terms, $G(1,y) > 0$.

\textbf{Case 2:} $x = 2$ and $y \in [3, \Delta-1]$.
By exact constraint evaluation, $G(2, \Delta) = 0$. We analyze the monotonicity of $G(2,y)$ by computing its partial derivative with respect to $y$:
\begin{equation*}
\frac{\partial G(2,y)}{\partial y} = -\frac{\partial f(2,y)}{\partial y} - \frac{y_2(\Delta)}{y^2} = -(3y^2 + 4y - 4) - \frac{y_2(\Delta)}{y^2}.
\end{equation*}
For $y \ge 3$, the polynomial component is bounded as $3y^2 + 4y - 4 \ge 27 + 12 - 4 = 35$. Referencing the extremal values of $y_2(\Delta)$ established at the beginning of the proof, the maximum magnitude of $-y_2(\Delta)$ is exactly $-y_2(8) = 66$. Combined with the minimum of $y^2 = 9$, we obtain $-y_2(\Delta)/y^2 \le 66/9 \approx 7.33$.
\begin{equation*}
\frac{\partial G(2,y)}{\partial y} \le -35 + 7.33 = -27.67 < 0.
\end{equation*}
Since the derivative is strictly negative, $G(2,y)$ is strictly decreasing. Because the function evaluates to zero at the upper boundary, it directly implies that for all internal points $y < \Delta$, the strict inequality $G(2,y) > G(2,\Delta) = 0$ holds.

\textbf{Case 3:} $x \ge 3$ and $y \in [4, \Delta]$.
We analyze the general partial derivative with respect to $y$:
\begin{equation*}
\frac{\partial G(x,y)}{\partial y} = -(3y^2 + 2xy - x^2) - \frac{y_2(\Delta)}{y^2}.
\end{equation*}
Let $P(x,y) = 3y^2 + 2xy - x^2$. Since $\frac{\partial P}{\partial x} = 2(y-x) > 0$, the function $P$ strictly increases with $x$. Bounding $x \ge 3$ and $y > x \implies y \ge 4$:
\begin{equation*}
P(x,y) \ge P(3,y) = 3y^2 + 6y - 9 \ge P(3,4) = 48 + 24 - 9 = 63.
\end{equation*}
Using the previously established maximum $-y_2(8) = 66$ and minimum $y^2 = 16$, we bound the fractional term as $-y_2(\Delta)/y^2 \le 66/16 = 4.125$, obtaining a strict upper bound:
\begin{equation*}
\frac{\partial G(x,y)}{\partial y} \le -63 + 4.125 = -58.875 < 0.
\end{equation*}
Thus, for any $x \ge 3$, the function strictly decreases with $y$, meaning its minimum must occur at $y = \Delta$. We now evaluate the behavior along this upper boundary $G(x, \Delta)$ by taking the partial derivative with respect to $x$:
\begin{equation*}
\frac{\partial G(x, \Delta)}{\partial x} = (3x^2 + 2x\Delta - \Delta^2) - \frac{y_2(\Delta)}{x^2}.
\end{equation*}
Let $Q(x, \Delta) = 3x^2 + 2x\Delta - \Delta^2$. Since $\frac{\partial Q}{\partial x} = 6x + 2\Delta > 0$, $Q$ is strictly increasing with $x$. For $x \ge 3$:
\begin{equation*}
Q(x, \Delta) \ge Q(3, \Delta) = 27 + 6\Delta - \Delta^2.
\end{equation*}
The quadratic function $27 + 6\Delta - \Delta^2$ is monotonically decreasing on $[3, 8]$ and achieves its minimum at $\Delta=8$, evaluating to $27 + 48 - 64 = 11$. For the independent fractional term, the minimum of $-y_2(\Delta)$ occurs at its maximum positive value, $y_2(3) = 14$, so $-y_2(\Delta)/x^2 \ge -14/9$. Thus:
\begin{equation*}
\frac{\partial G(x, \Delta)}{\partial x} \ge 11 - \frac{14}{9} = \frac{85}{9} > 0.
\end{equation*}
Because $\frac{\partial G(x, \Delta)}{\partial x} > 0$, the boundary function $G(x, \Delta)$ strictly increases with $x$. Its absolute minimum over $x \ge 3$ occurs exactly at $x = 3$. Evaluating $G(3, \Delta)$ directly yields a downward-opening parabola:
\begin{equation*}
G(3, \Delta) = -\frac{2}{3}\Delta^2 + 4\Delta + \frac{50}{3}.
\end{equation*}
This parabola takes its absolute minimum on the interval $[3, 8]$ at $\Delta = 8$, evaluating to $G(3, 8) = \frac{18}{3} = 6 > 0$. 

Since $G(x,y) > 0$ strictly holds across all three mutually exclusive and exhaustive cases, the proof is complete.
\end{proof}

With the global positivity of the $SO_{11}$ gap function analytically verified for $\Delta \le 8$, the integer linear programming conditions are strictly satisfied, leading to the following consequence."

\begin{corollary}[$SO_{11}$ Index]
\label{cor:so11}
Let $T$ be a tree of order $n$ with maximum degree $3 \le \Delta \le 8$. The maximum value of the $SO_{11}$ index is uniquely achieved by the extremal tree strictly restricted to the optimal edge basis $\{(1, \Delta), (2, \Delta)\}$.
\end{corollary}

\begin{proof}
According to the general optimality criterion established in Lemma \ref{lem:general-f-condition}, the optimal tree configuration uniquely maximizes a given topological index if the corresponding gap function satisfies $G(x,y) > 0$ for all valid internal pairs $(x,y) \notin \{(1, \Delta), (2, \Delta)\}$. By Lemma \ref{lem:so11_gap}, this condition is strictly satisfied for the $SO_{11}$ index across the entire discrete domain $1 \le x < y \le \Delta$ for any $\Delta \le 8$.
\end{proof}

\begin{remark}
The analytical evaluation of $G(3, \Delta)$ reveals precisely why $\Delta \le 8$ is a sharp constraint for the $SO_{11}$ index. At $\Delta = 9$, the gap function evaluates to $G(3, 9) = -\frac{4}{3} < 0$. Within the linear programming framework, this negative penalty guarantees that a tree strictly restricted to the standard basis $\{(1, \Delta), (2, \Delta)\}$ is no longer optimal. Instead, introducing non-basic edges of type $(3, 9)$ yields a strict increase in the objective function. This suboptimality requires a shift to a new configuration, which we formalize as a degree-dependent basis in Section \ref{sec:degree_dependent_basis_so11}.
\end{remark}

While the Sigma index was recently investigated \cite{basic2026sigma_arxiv, knor2026complete}, our framework effortlessly yields a highly concise alternative proof.

\begin{corollary}[Sigma Index]
\label{cor:sigma}
Let $T$ be a tree of order $n$ with maximum degree $\Delta \ge 3$. The maximum value of the Sigma index is uniquely achieved by the extremal tree strictly restricted to the optimal edge basis $\{(1, \Delta), (2, \Delta)\}$.
\end{corollary}

\begin{proof}
The weight function $f(x,y) = (y-x)^2$ trivially satisfies Lemma \ref{lem:reduction} since $\frac{\partial f}{\partial y} = 2(y-x) > 0$ for $x < y$. Evaluating the dual penalty coefficient yields:
\begin{equation*}
y_2 = 2\big( (\Delta - 1)^2 - (\Delta - 2)^2 \big) = 4\Delta - 6.
\end{equation*}
Substituting $f(x, \Delta) = (\Delta - x)^2$ and $y_2$ into \eqref{eq:test_hx}, we obtain:
\begin{equation*}
h(x) = (\Delta - 1)^2 - (\Delta - x)^2 - (4\Delta - 6)\frac{x-1}{x}.
\end{equation*}
Expanding the terms and reducing to a common denominator condenses the expression into a single rational function, which gracefully factors into a closed form:
\begin{equation}
h(x) = \frac{-x^3 + 2\Delta x^2 - (6\Delta - 7)x + 4\Delta - 6}{x} = \frac{(x-1)(x-2)(2\Delta - 3 - x)}{x}.
\end{equation}
For any valid interior degree $x \in [3, \Delta]$, the factors $(x-1)$, $(x-2)$, and $x$ are strictly positive. The remaining linear factor is bounded as $2\Delta - 3 - x \ge 2\Delta - 3 - \Delta = \Delta - 3 \ge 0$. Consequently, $h(x) \ge 0$ unconditionally holds for all valid $x$, satisfying the conditions of Theorem \ref{thm:main_1D}.
\end{proof}

\begin{corollary}[Albertson Index]
\label{cor:albertson}
Let $T$ be a tree of order $n$ with maximum degree $\Delta \ge 3$. The maximum value of the Albertson index is uniquely achieved by the extremal tree strictly restricted to the optimal edge basis $\{(1, \Delta), (2, \Delta)\}$.
\end{corollary}

\begin{proof}
The weight function $f(x,y) = y - x$ trivially satisfies Lemma \ref{lem:reduction} since $\frac{\partial f}{\partial y} = 1 > 0$. Evaluating the dual penalty coefficient yields:
\begin{equation*}
y_2 = 2\big( (\Delta - 1) - (\Delta - 2) \big) = 2.
\end{equation*}
Substituting $f(x, \Delta) = \Delta - x$ and $y_2 = 2$ into \eqref{eq:test_hx}, we obtain:
\begin{equation*}
h(x) = (\Delta - 1) - (\Delta - x) - 2\frac{x-1}{x}.
\end{equation*}
By cancelling the $\Delta$ terms and factoring the expression over a common denominator, the function gracefully condenses into a concise closed form:
\begin{equation}
h(x) = x - 1 - \frac{2(x-1)}{x} = \frac{(x-1)(x-2)}{x}.
\end{equation}
For any valid interior degree $x \in [3, \Delta]$, the factors $(x-1)$, $(x-2)$, and $x$ are strictly positive. Consequently, $h(x) > 0$ unconditionally holds for all valid $x$, satisfying the conditions of Theorem \ref{thm:main_1D}.
\end{proof}

\section{Alternative Optimality of the Basis $\{(1, \Delta), (\Delta, \Delta)\}$}
\label{sec:alternative_opt}

Theorem \ref{thm:main_1D} successfully dictates the extremal topology for the vast majority of indices, provided $\Delta$ is sufficiently large. However, as demonstrated by the $AG$ index constraint ($\Delta \ge 9$), the standard basis mathematically collapses for small maximum degrees. In this section, we address this boundary behavior by establishing an alternative optimality criterion.

For certain topological indices evaluated at smaller maximum degrees (e.g., chemical trees where $\Delta = 3$), the maximum is strictly achieved by the alternative basis $\{(1, \Delta), (\Delta, \Delta)\}$. In such trees, every internal vertex necessarily attains the maximum degree $\Delta$. 

One might intuitively attempt to realize this configuration by evaluating the universal bipartite class $\mathcal{T}(k, \Delta, q)$ at the limit $k = \Delta$. However, this substitution is structurally incomplete. When the partitions $V_k$ and $V_\Delta$ merge, the bipartite cross-edges $m_{k, \Delta}$ collapse into diagonal self-edges $m_{\Delta, \Delta}$. In the degree sum incidence identities \eqref{eq:tree-ni-i}, diagonal edges carry a coefficient of 2 (connecting two vertices of the same degree), whereas cross-edges carry a coefficient of 1. This discrete structural collapse alters the underlying Diophantine system, allowing internal vertices to connect freely rather than strictly bipartitely. Consequently, the alternative basis generates a much broader internal topology that necessitates an independent optimality criterion.

\begin{lemma}[Alternative Optimality Criterion]
\label{lem:alternative-f-condition}
Let $\Delta \ge 3$ and $n \ge \Delta+2$. Consider the linear program consisting of maximizing \eqref{eq:LP-objective-g} subject to \eqref{eq:LP-constraint-sum-SO} and \eqref{eq:LP-constraint-w-SO}. This problem admits a unique optimal solution mapped to the basis $m_{1, \Delta}$ and $m_{\Delta, \Delta}$, where:
\begin{equation}
m_{1, \Delta} = \frac{(\Delta-2)n + 2}{\Delta-1}, \qquad m_{\Delta, \Delta} = \frac{n-\Delta-1}{\Delta-1},
\end{equation}
if and only if for all valid non-basic pairs $(i,j) \notin \{(1, \Delta), (\Delta, \Delta)\}$, the weight function satisfies $G^*(i,j) > 0$, where the alternative gap function is defined as:
\begin{equation}
G^*(i,j) = f(\Delta, \Delta) - f(i,j) + y_2^* \left( \frac{1}{i} + \frac{1}{j} - \frac{2}{\Delta} \right),
\end{equation}
and the alternative dual penalty coefficient is $y_2^* = \frac{\Delta}{\Delta - 1} \big( f(1, \Delta) - f(\Delta, \Delta) \big)$.
\end{lemma}

\begin{proof}
We associate dual variables $y_1$ and $y_2^*$ with the constraints \eqref{eq:LP-constraint-sum-SO} and \eqref{eq:LP-constraint-w-SO}. By strict complementary slackness, for the unique optimal basic feasible solution where $m_{1, \Delta} > 0$ and $m_{\Delta, \Delta} > 0$, the corresponding dual constraints must be exact equalities:
\begin{align}
y_1 + \left(1 + \frac{1}{\Delta}\right)y_2^* &= f(1, \Delta), \label{eq:dual_alt1}\\
y_1 + \left(\frac{2}{\Delta}\right)y_2^* &= f(\Delta, \Delta). \label{eq:dual_alt2}
\end{align}
Subtracting \eqref{eq:dual_alt2} from \eqref{eq:dual_alt1} directly yields $y_2^* = \frac{\Delta}{\Delta - 1} \big( f(1, \Delta) - f(\Delta, \Delta) \big)$. For non-basic variables $m_{i,j}$, dual feasibility requires $y_1 + \left(\frac{1}{i}+\frac{1}{j}\right)y_2^* > f(i,j)$. Substituting $y_1 = f(\Delta, \Delta) - \left(\frac{2}{\Delta}\right)y_2^*$ into this inequality yields exactly $G^*(i,j) > 0$. 

By explicitly treating non-basic variables as zero, we can solve the equality constraints for the basis:
\begin{align*}
m_{1,\Delta} + m_{\Delta,\Delta} &= n - 1, \\
\left(1 + \frac{1}{\Delta}\right)m_{1,\Delta} + \left(\frac{2}{\Delta}\right)m_{\Delta,\Delta} &= n.
\end{align*}
Substituting $m_{\Delta,\Delta} = n - 1 - m_{1,\Delta}$ into the second equation algebraically isolates $m_{1, \Delta} = \frac{(\Delta-2)n + 2}{\Delta-1}$, which immediately determines $m_{\Delta, \Delta} = \frac{n-\Delta-1}{\Delta-1}$.
\end{proof}

Before applying this criterion, we characterize the exact parameters under which a tree realizes this collapsed basis without penalty. A tree restricted strictly to $\{(1, \Delta), (\Delta, \Delta)\}$ contains exclusively leaves ($n_1$) and vertices of maximum degree ($n_\Delta$). Substituting this into \eqref{eq:tree-sum-ni} and \eqref{eq:tree-sum-deg} yields:
\begin{align*}
n_1 + n_\Delta &= n, \\
n_1 + \Delta n_\Delta &= 2n - 2.
\end{align*}
Subtracting these equations directly isolates $n_\Delta$:
\begin{equation}
n_\Delta = \frac{n - 2}{\Delta - 1}, \qquad n_1 = \frac{n(\Delta - 2) + 2}{\Delta - 1}.
\end{equation}
For $n_\Delta$ to be an integer, the tree order must satisfy $n \equiv 2 \pmod{\Delta - 1}$. Parameterizing the internal vertices as $n_\Delta = p$ for some integer $p \ge 2$, the exact order becomes $n = p(\Delta - 1) + 2$. The parameter $p$ is strictly bounded below by 2 because the premise of Lemma \ref{lem:alternative-f-condition} dictates $n \ge \Delta+2$, explicitly excluding the star graph $K_{1, \Delta}$. The remaining parameters follow immediately:
\begin{equation}
n_1 = m_{1, \Delta} = p(\Delta - 2) + 2, \qquad m_{\Delta, \Delta} = p - 1.
\end{equation}

To formalize the topological existence of such graphs, we define the corresponding unipartite-internal tree class.

\begin{definition}
Let $p \ge 2$. We define $\mathcal{T}^\Delta_{\mathrm{opt}}(p)$ as the family of all trees constructed by the following procedure: Take an arbitrary underlying tree $H$ on $p$ vertices. For each vertex $v \in V(H)$, attach exactly $\Delta - d_H(v)$ pendant vertices to it, where $d_H(v)$ is the degree of $v$ in $H$.
\end{definition}

By construction, any tree in $\mathcal{T}^\Delta_{\mathrm{opt}}(p)$ has exactly $p$ internal vertices, all of which attain degree $\Delta$. Its internal edges correspond directly to the edges of $H$, yielding $m_{\Delta, \Delta} = p-1$, while the total number of attached pendant edges evaluates exactly to:
\begin{equation*}
m_{1, \Delta} = \sum_{v \in V(H)} \big(\Delta - d_H(v)\big) = p\Delta - 2(p-1) = p(\Delta - 2) + 2.
\end{equation*}

Conversely, any tree strictly restricted to the basis $\{(1, \Delta), (\Delta, \Delta)\}$ contains only internal vertices of degree $\Delta$. Because the tree is connected, these $p$ maximum degree vertices must induce a connected subtree $H$. Thus, the graph can only be formed by attaching leaves to $H$ until all its vertices reach degree $\Delta$, completely characterizing all valid tree configurations for $n \equiv 2 \pmod{\Delta - 1}$.

An illustrative example of this structural collapse is provided in Figure \ref{fig:optimal_alt}.

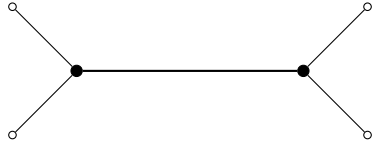
\begin{figure}[htpb]
    \centering
    \begin{tikzpicture}[
        v3/.style={circle, draw=black, fill=black, inner sep=1.5pt},
        leaf/.style={circle, draw=black, fill=white, inner sep=1pt},
        edge/.style={thick}
    ]
    
    \node[v3] (D1) at (-1.5, 0) {};
    \node[v3] (D2) at (1.5, 0) {};
    
    \draw[edge] (D1) -- (D2);
    
    \node[leaf] (L11) at ([shift={(135:1.2)}]D1) {};
    \node[leaf] (L12) at ([shift={(225:1.2)}]D1) {};
    \draw (D1) -- (L11);
    \draw (D1) -- (L12);
    
    \node[leaf] (L21) at ([shift={(45:1.2)}]D2) {};
    \node[leaf] (L22) at ([shift={(315:1.2)}]D2) {};
    \draw (D2) -- (L21);
    \draw (D2) -- (L22);
    
    \end{tikzpicture}
\caption{An extremal chemical tree belonging to the alternative class $\mathcal{T}^3_{\mathrm{opt}}(2)$ for the basis $\{(1, 3), (3, 3)\}$. Illustrated is the configuration for $\Delta=3$ and $p=2$, yielding order $n=6$. The internal maximum degree vertices (solid circles) are adjacent via the edge $(3,3)$, eliminating the bipartite structure.}
\label{fig:optimal_alt}
\end{figure}

\begin{lemma}[[Positivity of the AG Gap Function]
\label{lem:ag_alt_gap}
For any integer maximum degree $3 \le \Delta \le 8$, the alternative gap function $G^*(x,y)$ for the Arithmetic-Geometric ($AG$) index is strictly positive for all valid integer pairs $1 \le x \le y \le \Delta$, excluding the basis variables $(1, \Delta)$ and $(\Delta, \Delta)$.
\end{lemma}

\begin{proof}
From Lemma \ref{lem:alternative-f-condition}, the continuous extension of the alternative gap function and its dual penalty coefficient are analytically defined as:
\begin{equation*}
G^*(x,y) = 1 - \frac{x+y}{2\sqrt{xy}} + y_2^*(\Delta) \left( \frac{1}{x} + \frac{1}{y} - \frac{2}{\Delta} \right), \qquad y_2^*(\Delta) = \frac{\Delta}{\Delta - 1} \left( \frac{\Delta+1}{2\sqrt{\Delta}} - 1 \right).
\end{equation*}

We first establish that the penalty coefficient is strictly positive. Since $\frac{\Delta+1}{2\sqrt{\Delta}} > 1$ for all $\Delta \ge 2$, it trivially follows that $y_2^*(\Delta) > 0$. To prove $G^*(x,y) > 0$ for all non-basic pairs, we partition the valid domain into two cases.

\textbf{Case 1: $x = y < \Delta$.} \\
The alternative gap function simplifies to:
\begin{equation*}
G^*(x,x) = 1 - 1 + y_2^*(\Delta) \left( \frac{2}{x} - \frac{2}{\Delta} \right) = 2y_2^*(\Delta) \frac{\Delta - x}{x\Delta}.
\end{equation*}
Given $y_2^*(\Delta) > 0$ and $x < \Delta$, the fractional multiplier is strictly positive, guaranteeing $G^*(x,x) > 0$.

\textbf{Case 2: $x < y \le \Delta$.} \\
We evaluate the monotonicity of $G^*(x,y)$ with respect to its second variable by computing the partial derivative:
\begin{equation*}
\frac{\partial G^*(x,y)}{\partial y} = -\frac{\partial f(x,y)}{\partial y} - \frac{y_2^*(\Delta)}{y^2} = -\frac{y-x}{4\sqrt{x}y^{3/2}} - \frac{y_2^*(\Delta)}{y^2}.
\end{equation*}
Since $y > x$ and $y_2^*(\Delta) > 0$, both terms are strictly negative, yielding $\frac{\partial G^*}{\partial y} < 0$. This ensures that for any fixed $x$, the penalty function $G^*(x,y)$ is strictly decreasing as $y$ increases. Consequently, the absolute minimum over $y$ occurs at the upper boundary, establishing the lower bound $G^*(x,y) \ge G^*(x, \Delta)$.

To prove global positivity, we formally evaluate $G^*(x, \Delta)$. By factoring the penalty coefficient to $y_2^*(\Delta) = \frac{\sqrt{\Delta}(\sqrt{\Delta}-1)}{2(\sqrt{\Delta}+1)}$, the continuous boundary evaluates as:
\begin{equation*}
G^*(x, \Delta) = 1 - \frac{x+\Delta}{2\sqrt{x\Delta}} + \frac{\sqrt{\Delta}(\sqrt{\Delta}-1)}{2(\sqrt{\Delta}+1)} \left( \frac{\Delta-x}{x \Delta} \right).
\end{equation*}
By establishing a common denominator and factoring out the positive root difference $(\sqrt{\Delta}-\sqrt{x})$, this expression is transformed into the closed-form product:
\begin{equation*}
G^*(x, \Delta) = \frac{(\sqrt{\Delta}-\sqrt{x})(\sqrt{x}-1)}{2\sqrt{\Delta} x (\sqrt{\Delta}+1)} \big[ (\sqrt{\Delta}+1)\sqrt{x} - \sqrt{\Delta}(\sqrt{\Delta}-1) \big].
\end{equation*}
For any strictly internal degree $1 < x < \Delta$, the fractional coefficient is strictly positive. Therefore, the sign of the gap function is entirely governed by the bracketed factor. To guarantee $G^*(x, \Delta) > 0$, the following strict inequality must hold:
\begin{equation*}
(\sqrt{\Delta}+1)\sqrt{x} - \sqrt{\Delta}(\sqrt{\Delta}-1) > 0 \implies \sqrt{x} > \frac{\sqrt{\Delta}(\sqrt{\Delta}-1)}{\sqrt{\Delta}+1}.
\end{equation*}
Let $R(\Delta) = \frac{\sqrt{\Delta}(\sqrt{\Delta}-1)}{\sqrt{\Delta}+1}$. The derivative of $R(\Delta)$, which is
$R'(\Delta) = \frac{(\sqrt{\Delta} + 1)^2 - 2}{2\sqrt{\Delta} (\sqrt{\Delta} + 1)^2}$, confirms it is strictly monotonically increasing. Over the specified parameter space $3 \le \Delta \le 8$, $R(\Delta)$ achieves its maximum at $\Delta = 8$, evaluating to $R(8) = \frac{8 - \sqrt{8}}{\sqrt{8} + 1} \approx 1.35$. Since the minimum possible internal degree is $x=2$, the analytical inequality $\sqrt{x} \ge \sqrt{2} \approx 1.414 > R(8) \ge R(\Delta)$ is unequivocally satisfied.

Because the boundary condition $G^*(x, \Delta)$ remains strictly positive, $G^*(x,y) > 0$ universally holds across the entire non-basic domain.
\end{proof}

With the strict global positivity of the alternative gap function established for $3 \le \Delta \le 8$, the integer linear programming framework guarantees that the maximum $AG$ index is achieved by strictly minimizing the discrete penalty function over the integer domain. In the following corollaries, we apply this optimization constraint to explicitly resolve the exact topological structures of maximal trees for $\Delta = 3$ (chemical trees) and $\Delta = 4$. By evaluating the discrete penalty bounds alongside the Diophantine incidence identities, we rigorously characterize the required optimal non-basic edges.

\begin{corollary}[AG Index for Chemical Trees]
\label{cor:ag_chemical}
Let $T$ be a chemical tree ($\Delta = 3$) of order $n \ge 5$. The maximum value of the Arithmetic-Geometric ($AG$) index is uniquely determined by the alternative basis $\{(1, 3), (3, 3)\}$. Specifically:
\begin{itemize}
    \item If $n$ is even, the optimal tree consists strictly of $m_{1,3}$ and $m_{3,3}$ edges.
    \item If $n$ is odd, the optimal integer configuration includes exactly one degree-2 vertex, strictly introducing $m_{2,3} = 2$ alongside the basis edges.
\end{itemize}
\end{corollary}

\begin{proof}
By Lemma \ref{lem:alternative-f-condition} and Lemma \ref{lem:ag_alt_gap}, the LP maximization for $\Delta=3$ maps to the alternative basis $\mathcal{B} = \{(1,3), (3,3)\}$. To construct the continuous upper bound, we evaluate the dual variables. From Lemma \ref{lem:ag_alt_gap}, the penalty coefficient for $\Delta=3$ is $y_2^*(3) = \frac{3}{2}\big(\frac{4}{2\sqrt{3}} - 1\big) = \sqrt{3} - \frac{3}{2}$. 

The binding dual constraint for the basis edge $(3,3)$ dictates $y_1^* + \frac{2}{3}y_2^* = f(3,3)$. Since $f(3,3) = 1$, isolating the first dual variable yields $y_1^* = 1 - \frac{2}{3}(\sqrt{3} - \frac{3}{2}) = 2 - \frac{2}{\sqrt{3}}$.

Applying the decomposition from Eq. \eqref{eq:general_ti_decomposition}, the parameters resolve to $A = y_1^* + y_2^* = \frac{1}{\sqrt{3}} + \frac{1}{2}$ and $B = -y_1^* = \frac{2}{\sqrt{3}} - 2$. Consequently, the objective function reduces to:
\begin{equation}
AG(T) = \left( \frac{1}{\sqrt{3}} + \frac{1}{2} \right)n + \left( \frac{2}{\sqrt{3}} - 2 \right) - \sum_{(i,j) \notin \{(1,3), (3,3)\}} G^*(i,j)\, m_{i,j}.
\end{equation}

If $n$ is even, the LP solutions $m_{1,3} = \frac{n+2}{2}$ and $m_{3,3} = \frac{n-4}{2}$ are integers. Setting all non-basic $m_{i,j} = 0$ results in a zero penalty, defining the optimal tree.

If $n$ is odd, a tree restricted to degrees 1 and 3 cannot exist. The graph must contain at least one vertex of degree 2 ($n_2 \ge 1$). From the incidence relation \eqref{eq:tree-ni-2}, the degree-2 edges satisfy $m_{1,2} + 2m_{2,2} + m_{2,3} = 2n_2$. Multiplying this identity by $G^*(2,3) \approx 0.018$ gives:
\begin{equation}
2n_2 G^*(2,3) = G^*(2,3)m_{1,2} + 2G^*(2,3)m_{2,2} + G^*(2,3)m_{2,3}.
\end{equation}

Explicit evaluation of the non-basic gap penalties confirms the inequalities $G^*(1,2) \approx 0.133 > G^*(2,3)$ and $G^*(2,2) \approx 0.077 > 2G^*(2,3) \approx 0.036$. This establishes a lower bound for the total penalty function:
\begin{align}
\sum_{(i,j) \notin \{(1,3), (3,3)\}} G^*(i,j) m_{i,j} &= G^*(1,2)m_{1,2} + G^*(2,2)m_{2,2} + G^*(2,3)m_{2,3} \nonumber \\
&\ge G^*(2,3)m_{1,2} + 2G^*(2,3)m_{2,2} + G^*(2,3)m_{2,3} \nonumber \\
&= 2n_2 G^*(2,3).
\end{align}

Since $n_2 \ge 1$, the minimum penalty is bounded below by $2G^*(2,3)$. This minimum is attained if and only if $n_2 = 1$, $m_{2,3} = 2$, and $m_{1,2} = m_{2,2} = 0$. Substituting $m_{2,3} = 2$ into the LP constraints shifts the basic variables to $m_{1,3} = \frac{n+1}{2}$ and $m_{3,3} = \frac{n-5}{2}$, generating a valid integer configuration.
\end{proof}

\begin{corollary}[AG Index for $\Delta=4$]
\label{cor:ag_delta4}
Let $T$ be a tree of order $n \ge 6$ with maximum degree $\Delta = 4$. The maximum value of the Arithmetic-Geometric ($AG$) index is uniquely determined by the alternative basis $\{(1, 4), (4, 4)\}$. Specifically:
\begin{itemize}
    \item If $n \equiv 2 \pmod 3$, the optimal tree consists solely of $m_{1,4}$ and $m_{4,4}$ edges.
    \item If $n \equiv 1 \pmod 3$, the optimal configuration includes exactly one degree-3 vertex, introducing $m_{3,4} = 3$ alongside the basis edges.
    \item If $n \equiv 0 \pmod 3$, the optimal configuration includes exactly one degree-2 vertex, introducing $m_{2,4} = 2$ alongside the basis edges.
\end{itemize}
\end{corollary}

\begin{proof}
By Lemma \ref{lem:alternative-f-condition}, the LP maximization for $\Delta=4$ maps to the alternative basis $\mathcal{B} = \{(1,4), (4,4)\}$. To construct the continuous upper bound, we evaluate the dual variables. The penalty coefficient evaluates to $y_2^*(4) = \frac{4}{3}(\frac{5}{4} - 1) = \frac{1}{3}$. 

The dual constraint for the basis edge $(4,4)$ dictates $y_1^* + \frac{2}{4}y_2^* = f(4,4)$. Since $f(4,4) = 1$, isolating $y_1^*$ yields $y_1^* = 1 - \frac{1}{2}(\frac{1}{3}) = \frac{5}{6}$.

Applying the decomposition from Eq. \eqref{eq:general_ti_decomposition}, the parameters resolve to $A = y_1^* + y_2^* = \frac{5}{6} + \frac{1}{3} = \frac{7}{6}$ and $B = -y_1^* = -\frac{5}{6}$. The objective function reduces to:
\begin{equation}
AG(T) = \frac{7}{6}n - \frac{5}{6} - \sum_{(i,j) \notin \{(1,4), (4,4)\}} G^*(i,j)\, m_{i,j}.
\end{equation}

The LP solutions are $m_{1,4} = \frac{2n+2}{3}$ and $m_{4,4} = \frac{n-5}{3}$. 

If $n \equiv 2 \pmod 3$, these solutions are integers. Setting all non-basic $m_{i,j} = 0$ results in a zero penalty, defining the optimal tree.

If $n \not\equiv 2 \pmod 3$, a tree restricted to degrees 1 and 4 cannot exist. The graph must contain internal vertices of degree $k \in \{2,3\}$. Evaluation of the non-basic gap penalties yields: $G^*(3,4) \approx 0.017$, $G^*(2,4) \approx 0.023$, $G^*(3,3) \approx 0.055$, $G^*(2,3) \approx 0.090$, $G^*(1,3) \approx 0.123$, $G^*(2,2) \approx 0.166$, and $G^*(1,2) \approx 0.272$. 

To minimize the penalty, we analyze the lower bound required to introduce these vertices using the incidence identities \eqref{eq:tree-ni-i}.

For any degree-3 vertex ($n_3 \ge 1$), the identity $m_{1,3} + m_{2,3} + 2m_{3,3} + m_{3,4} = 3n_3$ holds. Since $G^*(1,3) > G^*(3,4)$, $G^*(2,3) > G^*(3,4)$, and $G^*(3,3) \approx 0.055 > 2G^*(3,4) \approx 0.034$, the penalty associated with degree-3 vertices is bounded below:
\begin{align}
\sum_{(i,3) \text{ or } (3,j)} G^*(i,j) m_{i,j} &= G^*(1,3)m_{1,3} + G^*(2,3)m_{2,3} + G^*(3,3)m_{3,3} + G^*(3,4)m_{3,4} \nonumber \\
&\ge G^*(3,4)m_{1,3} + G^*(3,4)m_{2,3} + 2G^*(3,4)m_{3,3} + G^*(3,4)m_{3,4} \nonumber \\
&= 3n_3 G^*(3,4).
\end{align}

Similarly, for any degree-2 vertex ($n_2 \ge 1$), the identity $m_{1,2} + 2m_{2,2} + m_{2,3} + m_{2,4} = 2n_2$ holds. Since $G^*(1,2) > G^*(2,4)$, $G^*(2,3) > G^*(2,4)$, and $G^*(2,2) \approx 0.166 > 2G^*(2,4) \approx 0.046$, the penalty contribution is bounded below:
\begin{align}
\sum_{(i,2) \text{ or } (2,j)} G^*(i,j) m_{i,j} &= G^*(1,2)m_{1,2} + G^*(2,2)m_{2,2} + G^*(2,3)m_{2,3} + G^*(2,4)m_{2,4} \nonumber \\
&\ge G^*(2,4)m_{1,2} + 2G^*(2,4)m_{2,2} + G^*(2,4)m_{2,3} + G^*(2,4)m_{2,4} \nonumber \\
&= 2n_2 G^*(2,4).
\end{align}

To determine which non-basic vertices must be introduced, we subtract the vertex sum identity ($n_1 + n_2 + n_3 + n_4 = n$) from the degree sum identity ($n_1 + 2n_2 + 3n_3 + 4n_4 = 2n-2$) to isolate $3n_4 = n - 2 - n_2 - 2n_3$. This yields the modular congruence:
\begin{equation}
\label{eq:diophantine-mod3}
n_2 + 2n_3 \equiv n - 2 \pmod 3.
\end{equation}

\textbf{Case $n \equiv 1 \pmod 3$:} 
Constraint \eqref{eq:diophantine-mod3} evaluates to $n_2 + 2n_3 \equiv 2 \pmod 3$. The minimal integer pairs satisfying this congruence are $(0, 1)$ and $(2, 0)$. 
If $(n_2, n_3) = (2, 0)$, the penalty is bounded below by $2 \times (2G^*(2,4)) \approx 0.092$.
If $(n_2, n_3) = (0, 1)$, the penalty is bounded below by $1 \times (3G^*(3,4)) \approx 0.051$.
Since $0.051 < 0.092$, the minimum penalty is attained by setting $n_3 = 1$ and $n_2 = 0$. Substituting $m_{3,4} = 3$ into the LP constraints shifts the basic variables to $m_{1,4} = \frac{2n+1}{3}$, generating a valid integer configuration.

\textbf{Case $n \equiv 0 \pmod 3$:} 
Constraint \eqref{eq:diophantine-mod3} evaluates to $n_2 + 2n_3 \equiv 1 \pmod 3$. The minimal integer pairs satisfying this relation are $(1, 0)$ and $(0, 2)$.
If $(n_2, n_3) = (0, 2)$, the lower penalty bound is $2 \times (3G^*(3,4)) \approx 0.102$.
If $(n_2, n_3) = (1, 0)$, the lower penalty bound is $1 \times (2G^*(2,4)) \approx 0.046$.
Since $0.046 < 0.102$, the minimum penalty is attained by setting $n_2 = 1$ and $n_3 = 0$. Substituting $m_{2,4} = 2$ resolves the basic constraints to $m_{1,4} = \frac{2n}{3}$, defining the minimal discrete configuration.
\end{proof}

\section{The Degree-Dependent Basis $\{(1, \Delta), (k, \Delta)\}$ for the $SO_{11}$ Index}
\label{sec:degree_dependent_basis_so11}

As established, neither the standard basis $\{(1, \Delta), (2, \Delta)\}$ nor the alternative basis $\{(1, \Delta), (\Delta, \Delta)\}$ universally maximizes the $SO_{11}$ index. Specifically, the standard basis fails for $\Delta \ge 9$. This limitation necessitates a degree-dependent optimal basis parameterized by $\{(1, \Delta), (k, \Delta)\}$. In this section, we formally derive the exact integer degree $k$ that uniquely maximizes the $SO_{11}$ index as a function of $\Delta$.

\begin{theorem}[Degree-Dependent Optimal Basis]
\label{thm:degree_dependent_k}
Let $T$ be a tree of order $n$ and maximum degree $\Delta \ge 3$. The linear program maximizing the $SO_{11}$ index admits an optimal basic feasible solution generated by the basis $\{(1, \Delta), (k, \Delta)\}$, where the optimal integer degree $k$ is uniquely determined by maximizing the characteristic polynomial:
\begin{equation}
M(k) = -k^3 - (\Delta+1)k^2 + (\Delta^2-\Delta-1)k.
\end{equation}
This integer solution is strictly bounded to $k \in \{\lfloor x^* \rfloor, \lceil x^* \rceil\}$, where $x^*$ is the continuous critical point evaluated as the positive root:
\begin{equation}
x^* = \frac{\sqrt{4\Delta^2 - \Delta - 2} - (\Delta + 1)}{3}.
\end{equation}
\end{theorem}

\begin{proof}
To establish the global optimality of the basis $\{(1, \Delta), (k, \Delta)\}$, we evaluate the dual variables $y_1$ and $y_2$. By strict complementary slackness:
\begin{align}
y_1 + \left(1 + \frac{1}{\Delta}\right)y_2 &= f(1, \Delta), \label{eq:dual_dyn1}\\
y_1 + \left(\frac{1}{k} + \frac{1}{\Delta}\right)y_2 &= f(k, \Delta). \label{eq:dual_dyn2}
\end{align}
Subtracting \eqref{eq:dual_dyn2} from \eqref{eq:dual_dyn1} isolates the dual penalty coefficient:
\begin{equation*}
\left(1 - \frac{1}{k}\right)y_2 = f(1, \Delta) - f(k, \Delta) \implies y_2 = \frac{k}{k-1} \big( f(1, \Delta) - f(k, \Delta) \big).
\end{equation*}
We define the characteristic function $M(k) = -y_2 = \frac{k}{k-1} \big( f(k, \Delta) - f(1, \Delta) \big)$. 
For the $SO_{11}$ index, $f(x,y) = y^3 + y^2 x - x^2 y - x^3$. The algebraic difference evaluates to:
\begin{align*}
f(k, \Delta) - f(1, \Delta) &= (\Delta^3 + k\Delta^2 - k^2\Delta - k^3) - (\Delta^3 + \Delta^2 - \Delta - 1) \\
&= (k-1)\Delta^2 - (k^2-1)\Delta - (k^3-1).
\end{align*}
Factoring out the term $(k-1)$ cleanly cancels the denominator in the definition of $M(k)$, yielding the cubic polynomial:
\begin{equation*}
M(k) = -k^3 - (\Delta+1)k^2 + (\Delta^2-\Delta-1)k.
\end{equation*}

The basis is strictly optimal if and only if the dual gap satisfies $G_k^*(x, \Delta) \ge 0$ for all valid degrees $x$. The gap function is defined as:
\begin{equation*}
G_k^*(x, \Delta) = f(1, \Delta) - f(x, \Delta) - y_2 \left( 1 - \frac{1}{x} \right).
\end{equation*}
Substituting $y_2 = -M(k)$ and factoring out $\frac{x-1}{x}$ yields:
\begin{align*}
G_k^*(x, \Delta) &= f(1, \Delta) - f(x, \Delta) + M(k) \left( \frac{x-1}{x} \right) \\
&= \frac{x-1}{x} \left[ \frac{x}{x-1} \big( f(1, \Delta) - f(x, \Delta) \big) + M(k) \right] \\
&= \frac{x-1}{x} \big[ M(k) - M(x) \big].
\end{align*}
For any valid degree $x > 1$, the fractional prefix $\frac{x-1}{x}$ is strictly positive. Therefore, the condition $G_k^*(x, \Delta) \ge 0$ inherently holds if and only if $M(k) \ge M(x)$ for all integer alternatives $x$. This is satisfied by selecting the optimal integer $k$ as the maximum of the polynomial $M$.

To determine this maximum analytically, we set the derivative of the continuous extension to zero, $M'(x) = -3x^2 - 2(\Delta+1)x + (\Delta^2-\Delta-1) = 0$. The global maximum corresponds to the positive root:
\begin{equation*}
x^* = \frac{\sqrt{4\Delta^2 - \Delta - 2} - (\Delta + 1)}{3}.
\end{equation*}
Since the optimal degree $k$ must be an integer, it is strictly bounded to the nearest integers $k \in \{\lfloor x^* \rfloor, \lceil x^* \rceil\}$.
\end{proof}

\begin{corollary}[Asymptotic Scaling and Basis Transition for $SO_{11}$]
The optimal basis for the $SO_{11}$ index strictly transitions from $k=2$ to $k=3$ at $\Delta = 9$. Asymptotically, as $\Delta \to \infty$, the optimal integer degree scales proportionally to $\Delta/3$.
\end{corollary}

\begin{proof}
By Theorem \ref{thm:degree_dependent_k}, the optimal integer degree $k$ is determined by maximizing $M(k)$ near the critical point $x^*$. We evaluate the polynomial at the bounding integers to identify the transition threshold:
\begin{itemize}
    \item For $\Delta = 8$: $x^* = \frac{\sqrt{246}-9}{3} \approx 2.23$. Evaluating the bounding integers yields $M(2) = 66$ and $M(3) = 57$. Since $M(2) > M(3)$, the basis $k=2$ remains optimal.
    \item For $\Delta = 9$: $x^* = \frac{\sqrt{313}-10}{3} \approx 2.56$. Evaluating the bounding integers yields $M(2) = 94$ and $M(3) = 96$. Since $M(3) > M(2)$, the optimal basis strictly shifts to $k=3$.
\end{itemize}

To determine the asymptotic scaling, we evaluate the limit of the continuous critical point as $\Delta \to \infty$:
\begin{equation*}
\lim_{\Delta \to \infty} \frac{x^*}{\Delta} = \lim_{\Delta \to \infty} \frac{\sqrt{4\Delta^2 - \Delta - 2} - (\Delta + 1)}{3\Delta} = \frac{2 - 1}{3} = \frac{1}{3}.
\end{equation*}
Consequently, $x^* \sim \frac{\Delta}{3}$. This rigorously demonstrates that for sufficiently large maximum degrees, the optimal basis parameter strictly diverges from the boundaries $k=2$ and $k=\Delta$, scaling precisely as a one-third fraction of the maximum degree.
\end{proof}

An illustrative example of this structural transition is provided in Figure \ref{fig:optimal_so11_delta9}. By utilizing the universal generative class $\mathcal{T}(k, \Delta, q)$ established in Section \ref{sec:applications}, we visually construct the extremal configuration for $\Delta = 9$, where the optimal basis shifts to $k=3$. Setting the parameter $q=1$ generates the minimal non-trivial configuration $T \in \mathcal{T}(3, 9, 1)$ that maximizes the $SO_{11}$ index. By directly applying the derived structural formulas, the exact vertex distribution evaluates to:
\begin{align*}
n_3 &= q = 1, \\
n_9 &= q(k-1) + 1 = 1(2) + 1 = 3, \\
n_1 &= q(\Delta k - \Delta - k) + \Delta = 1(27 - 9 - 3) + 9 = 24.
\end{align*}
Summing these vertices confirms the tree order $n = 28$. As depicted, the single internal vertex of degree 3 connects to three maximum degree vertices, which together anchor the twenty-four leaves.

\begin{figure}[htpb]
    \centering
    \begin{tikzpicture}[
        v3/.style={rectangle, draw=black, fill=gray!30, inner sep=2.5pt, minimum size=6pt},
        v9/.style={circle, draw=black, fill=black, inner sep=1.5pt},
        leaf/.style={circle, draw=black, fill=white, inner sep=1pt},
        edge/.style={thick}
    ]
    
    \node[v3] (C) at (0,0) {};
    
    \node[v9] (A) at (90:1.5) {};
    \node[v9] (B) at (210:1.5) {};
    \node[v9] (Cnode) at (330:1.5) {};
    
    \draw[edge] (C) -- (A);
    \draw[edge] (C) -- (B);
    \draw[edge] (C) -- (Cnode);
    
    \foreach \angle/\name in {90/A, 210/B, 330/Cnode} {
        \foreach \i in {1,2,...,8} {
            \pgfmathsetmacro{\leafangle}{\angle - 52.5 + \i*15}
            \node[leaf] (L\name\i) at (\leafangle:3.2) {};
            \draw (\name) -- (L\name\i);
        }
    }
    
    \end{tikzpicture}
\caption{The minimal non-trivial extremal tree $T \in \mathcal{T}(3, 9, 1)$ of order $n=28$ with maximum degree $\Delta=9$. The internal structure is a star $K_{1,3}$ connecting the single degree-3 vertex in $V_3$ (square) to three degree-9 vertices in $V_9$ (solid circles). This exactly satisfies the optimal basis $\{(1,9), (3,9)\}$, uniquely maximizing the $SO_{11}$ index. Pendant leaves ($V_1$) are represented by hollow circles.}
\label{fig:optimal_so11_delta9}
\end{figure}
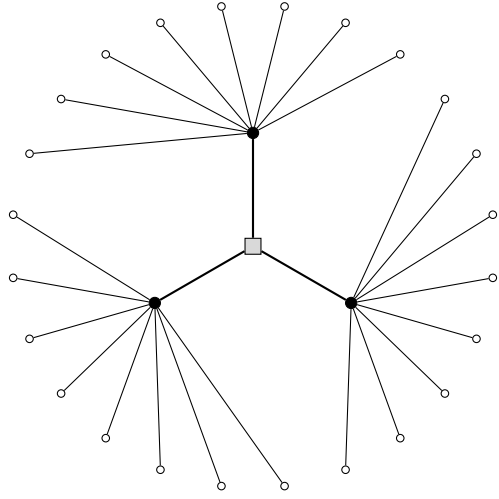

\section{Concluding Remarks}
\label{sec:conclusion}

In this paper, we established a unified linear programming approach that replaces index-specific combinatorial proofs with a deterministic algebraic methodology. By formulating the optimization over degree-pair multiplicities and exploiting primal-dual strict complementary slackness, we proved that maximizing a vertex-degree-based index over trees strictly reduces to minimizing a discrete dual penalty function.

We derived a continuous, single-variable test condition that guarantees the global optimality of the standard edge basis $\{(1, \Delta), (2, \Delta)\}$. This criterion systematically resolved the extremal trees for the Tangent, $SO_1$, $SO_2$, $\sigma$-irregularity, Albertson, and $AG$ indices.

Furthermore, the dual optimality conditions establish exact analytical thresholds dictating necessary basis shifts. By introducing the bipartite generative class $\mathcal{T}(k, \Delta, q)$, we explicitly characterized the topologies resulting from these transitions. This allowed us to analytically resolve the structural collapse into the alternative basis $\{(1, \Delta), (\Delta, \Delta)\}$ for the $AG$ index on small maximum degrees ($\Delta \in \{3, 4\}$). For the $SO_{11}$ index, we proved that while the standard basis remains optimal exclusively for $\Delta \le 8$, the configuration strictly transitions to a degree-dependent basis $\{(1, \Delta), (k, \Delta)\}$ for $\Delta \ge 9$, where the optimal internal degree $k$ scales asymptotically as $\Delta/3$ based on the maximum of a derived cubic polynomial.

This framework provides an exact analytical pathway for any topological index whose extremal configuration parameterizes within $\mathcal{T}(k, \Delta, q)$ or its unipartite forms. Future research can systematically extend this dual penalty mechanism to graph classes beyond trees, such as unicyclic or bicyclic graphs with bounded maximum degrees.

\section*{Acknowledgements}
The research has been supported by the research project of the Ministry of Education, Science and Technological Development of the Republic of Serbia (No.~451-03-34/2026-03/200124).



\begin{thebibliography}{99}

\bibitem{abdo2018sigma}
H. Abdo, D. Dimitrov, I. Gutman, 
\textit{Graphs with maximal $\sigma$-irregularity}, 
Discrete Appl. Math. 250 (2018) 57--64.

\bibitem{albertson1997}
M. O. Albertson,
\textit{The irregularity of a graph},
Ars Combin. 46 (1997) 219--225.

\bibitem{ali2023solutions}
A. Ali, I. \v{Z}. Milovanovi\'{c}, A. M. Albalahi, A. M. Alanazi, A. E. Hamza,
\textit{Solutions to some open problems about four Sombor-index-like graph invariants}, 
MATCH Commun. Math. Comput. Chem. 90 (2023) 505--512.

\bibitem{basic2024filomat}
M. Ba\v{s}i\'{c}, 
\textit{Maximum values of the Sombor-index-like graph invariants of trees and connected graphs}, 
Filomat 38 (26) (2024) 9043--9053.

\bibitem{basic2026sigma_arxiv}
M. Ba\v{s}i\'{c},
\textit{Trees with maximum $\sigma$-irregularity under a prescribed maximum degree 6},
arXiv:2602.01262 (2026).

\bibitem{basic2026sigma_arxiv}
M. Ba\v{s}i\'{c}, 
\textit{On the maximum $\sigma$-irregularity of trees with given order and maximum degree}, 
arXiv:2602.14241 (2026).

\bibitem{basic2026filomat}
M. Ba\v{s}i\'{c},
\textit{Extremal results for the general reduced Second Zagreb index and Sombor-Type indices on trees}, 
Filomat 40 (13) (2026).


\bibitem{basic2026cmj}
M. Ba\v{s}i\'{c}, 
\textit{Extremal values of the Sombor-index-like invariant $SO_6$ for trees with given order and maximum degree}, 
Czechoslovak Math. J. (2026), \url{https://doi.org/10.21136/CMJ.2026.0199-26}.

\bibitem{ag_amc_reference}
K. C. Das, I. Gutman, B. Furtula, 
\textit{On the first geometric-arithmetic index of graphs}, 
Appl. Math. Comput. 159 (2011) 2030--2037.

\bibitem{gao2024extremal}
W. Gao,
\textit{Extremal trees and molecular trees with respect to the Sombor-index-like graph invariants $\mathcal{SO}_5$ and $\mathcal{SO}_6$}, 
Czechoslovak Math. J. 74 (2024) 927--941.

\bibitem{gutman2021sombor}
I. Gutman, 
\textit{Geometric approach to degree-based topological indices: Sombor indices},
MATCH Commun. Math. Comput. Chem. 86 (1) (2021) 11--16.

\bibitem{gutman2022back}
I. Gutman,
\textit{Sombor indices -- back to geometry},
Open J. Discrete Appl. Math. 5 (2) (2022) 1--5.

\bibitem{gutman2018inverse}
I. Gutman, M. Togan, A. Yurtta\c{s}, A. S. \c{C}evik, I. N. Cang\"{u}l, 
\textit{Inverse problem for sigma index}, 
MATCH Commun. Math. Comput. Chem. 79 (2018) 491--508.

\bibitem{knor2026complete}
M. Knor, J. Sedlar, R. \v{S}krekovski, 
\textit{A complete characterization of maximally $\sigma$-irregular trees with prescribed maximum degree}, 
arXiv:2608.07402 (2026).

\bibitem{kovijanic2024chemical}
\v{Z}. Kovijani\'{c} Vuki\'{c}evi\'{c}, G. Popivoda, S. Vujo\v{s}evi\'{c}, R. \v{S}krekovski, D. Dimitrov, 
\textit{The $\sigma$-irregularity of chemical trees}, 
MATCH Commun. Math. Comput. Chem. 91 (2024) 267--282.

\bibitem{kovijanic2026delta5}
\v{Z}. Kovijani\'{c} Vuki\'{c}evi\'{c}, G. Popivoda, S. Vujo\v{s}evi\'{c}, R. \v{S}krekovski, D. Dimitrov, J. Sedlar, 
\textit{The $\sigma$-irregularity of trees with maximum degree 5}, 
Discrete Appl. Math. 382 (2026) 124--136.

\bibitem{oz2024novel}
M. S. \"{O}z,
\textit{Novel Sombor-index-like topological indices with geometry and calculus},
Sci. Publ. State Univ. Novi Pazar Ser. A Appl. Math. Inform. Mech. 16 (2) (2024) 147--157.

\bibitem{oz2025more} 
M. S. \"{O}z, I. Gutman, 
\textit{More geometric studies of vertex-degree-based graph indices--tangent Sombor index}, 
Hacettepe J. Math. Stat. 54 (6) (2025) 2195--2205.

\bibitem{rada2021general}
J. Rada, J. M. Rodriguez, J. M. Sigarreta,
\textit{General properties on Sombor indices}, 
Discrete Appl. Math. 299 (2021) 87--97.

\bibitem{ravindrana2025review}
B. A. Ravindran,
\textit{Sombor index on trees -- A review}, 
MATCH Commun. Math. Comput. Chem. 94 (2) (2025) 291--305.

\bibitem{tang2023trees}
Z. Tang, Q. Li, H. Deng, 
\textit{Trees with extremal values of the Sombor-index-like graph invariants}, 
MATCH Commun. Math. Comput. Chem. 90 (1) (2023) 203--222.

\end{thebibliography}
\end{document}